\documentclass{birkjour}
\usepackage{amsmath,amsthm,amscd,amssymb,latexsym,upref,mathrsfs}
\usepackage{color}

\usepackage{hyperref}
\newcommand*{\mailto}[1]{\href{mailto:#1}{\nolinkurl{#1}}}

\newcommand{\bbN}{{\mathbb{N}}}
\newcommand{\bbR}{{\mathbb{R}}}

\newcommand{\bbC}{{\mathbb{C}}}

\newcommand{\dC}{{\mathbb{C}}}
\newcommand{\dR}{{\mathbb{R}}}

\newcommand{\cC}{{\mathcal C}}
\newcommand{\cD}{{\mathcal D}}
\newcommand{\cE}{{\mathcal E}}

\newcommand{\cG}{{\mathcal G}}

\newcommand{\cL}{{\mathcal L}}
\newcommand{\cM}{{\mathcal M}}

\newcommand{\sH}{{\mathfrak H}}
\newcommand{\sS}{{\mathfrak S}}

\newcommand{\lb}{\label}

\newcommand{\ol}{\overline}

\newcommand{\f}{\frac}

\def\senki{{\lbrack\negthinspace [\bot ]\negthinspace\rbrack}}
\def\senki+{{\lbrack\negthinspace [+] \negthinspace\rbrack}}

\renewcommand{\Im}{\mathop\mathrm{Im}}
\renewcommand{\ge}{\geqslant}

\DeclareMathOperator{\dom}{dom}
\DeclareMathOperator{\ran}{ran}
\DeclareMathOperator{\tr}{tr}

\allowdisplaybreaks \numberwithin{equation}{section}

\newtheorem{theorem}{Theorem}[section]

\newtheorem{lemma}[theorem]{Lemma}
\newtheorem{corollary}[theorem]{Corollary}
\newtheorem{definition}[theorem]{Definition}
\newtheorem{hypothesis}[theorem]{Hypothesis}

\theoremstyle{remark}

\begin{document}

\title[Spectral shift function for Schr\"{o}dinger operators]{A spectral shift function for Schr\"{o}dinger operators with singular interactions}

\author[J.\ Behrndt]{Jussi Behrndt}
\address{Institut f\"ur Numerische Mathematik, Technische Universit\"at
Graz, Steyrergasse 30, 8010 Graz, Austria}
\email{\mailto{behrndt@tugraz.at}}
\urladdr{\url{http://www.math.tugraz.at/~behrndt/}}

\author[F.\ Gesztesy]{Fritz Gesztesy} 
\address{Department of Mathematics
Baylor University, One Bear Place \#97328,
Waco, TX 76798-7328, USA}
\email{\mailto{Fritz\_Gesztesy@baylor.edu}}
\urladdr{\url{http://www.baylor.edu/math/index.php?id=935340}}

\author[S.\ Nakamura]{Shu Nakamura}
\address{Graduate School of Mathematical Sciences, University of Tokyo, 
3-8-1, Komaba, Meguro-ku, Tokyo, Japan 153-8914}
\email{\mailto{shu@ms.u-tokyo.ac.jp}}
\urladdr{\url{http://www.ms.u-tokyo.ac.jp/~shu/}}

\date{\today}



\begin{abstract}
For the pair $\{-\Delta, -\Delta-\alpha\delta_\cC\}$ of self-adjoint Schr\"{o}dinger operators in $L^2(\dR^n)$ a spectral shift function 
is determined in an explicit form with the help of (energy parameter 
dependent) Dirichlet-to-Neumann maps. Here $\delta_\cC$ denotes a singular $\delta$-potential which is supported on a smooth compact hypersurface $\cC\subset\dR^n$
and $\alpha$ is a real-valued function on $\cC$.
\end{abstract}

\maketitle


\section{Introduction}

The goal of this paper is to determine a spectral shift function for the pair $\{H,H_{\delta,\alpha}\}$, where $H=-\Delta$ is the usual self-adjoint Laplacian
in $L^2(\dR^n)$, and $H_{\delta,\alpha}=-\Delta-\alpha\delta_\cC$ is a singular perturbation of $H$ by a $\delta$-potential of variable real-valued 
strength $\alpha\in C^1(\cC)$ supported on some smooth,
compact hypersurface $\cC$ that splits $\dR^n$, $n\geq 2$, into a bounded interior and an unbounded exterior domain. 
Schr\"{o}dinger operators with $\delta$-interactions are often used as idealized models 
of physical systems with short-range potentials; in the simplest case point interactions are considered, but in the last decades also 
interactions supported on curves and hypersurfaces have attracted a lot of attention, see the monographs \cite{AGHH05,AK99,EK15}, the review \cite{E08}, and, for instance, 
\cite{AKMN13,AGS87,BLL13-AHP,BLL-Exner,BMN17,BEKS94,EI01,EK03,EK05,EY02,MPS15} for a small 
selection of papers in this area. 

It is known from \cite{BLL13-AHP} (see also \cite{BLL-Exner}) that for an integer $m>(n/2)-1$ the $m$-th powers of the resolvents of $H$ and $H_{\delta,\alpha}$ differ by a trace class operator,
\begin{equation}\label{resabmmm}
\big[(H_{\delta,\alpha} - z I_{L^2(\dR^n)})^{-m}-(H- z I_{L^2(\dR^n)})^{-m}\big]\in\sS_1(L^2(\dR^n)).
\end{equation}
Since both operators, $H$ and $H_{\delta,\alpha}$ are bounded from below, 
\cite[Theorem~8.9.1, p.~306--307]{Y92} applies (upon replacing the pair $(H, H_{\delta,\alpha})$ by $H + C I_{L^2(\bbR^n)}, 
H_{\delta,\alpha} + C I_{L^2(\bbR^n)})$ such that $H + C I_{L^2(\bbR^n)} \geq  I_{L^2(\bbR^n)}$ and 
$H_{\delta,\alpha} + C I_{L^2(\bbR^n)} \geq  I_{L^2(\bbR^n)}$ for some $C > 0$) and there exists a real-valued function 
$\xi\in L^1_{\rm loc}(\dR)$ satisfying 
\begin{equation*}
\int_\dR \f{\vert\xi(\lambda)\vert \, d\lambda}{(1+\vert \lambda \vert)^{m+1}}  < \infty,  
\end{equation*} 
such that the trace formula
\begin{equation*}
\tr_{L^2(\dR^n)}\bigl((H_{\delta,\alpha}- z I_{L^2(\dR^n)})^{-m}-(H- z I_{L^2(\dR^n)})^{-m}\bigr) 
= -m \int_\dR \frac{\xi(\lambda)\,d\lambda}{(\lambda - z)^{m+1}}
\end{equation*}
is valid for all $z\in\rho(H_{\delta,\alpha})\cap\rho(H)$.
The function $\xi$ in the integrand on the right-hand side is called a {\it spectral shift function} of the pair $\{H,H_{\delta,\alpha}\}$.
For more details, the history and developments of the spectral shift function we refer the reader to the 
survey papers 
\cite{BP98,BY92,BY92-1}, the standard monographs \cite{Y92,Y10}, the paper \cite{Y05}, and the original works \cite{L52,L56} by I.\,M. Lifshitz, \cite{K53,K62} by M.\,G.~Krein. 

Our approach in this note is based on techniques from extension theory of symmetric operators and relies on a recent representation result of the spectral shift function
in terms of an abstract Titchmarsh-Weyl $m$-function from \cite{BGN17}, which we recall in Section~\ref{ssfsec} for the convenience of the reader.
In our situation this abstract Titchmarsh--Weyl $m$-function will turn out to be a combination of energy dependent  Dirichlet-to-Neumann 
maps $\cD_{\rm i}(z)$ and $\cD_{\rm e}(z)$ associated to $-\Delta$ on the interior and exterior domain, respectively.
More precisely, we shall interpret $H$ and $H_{\delta,\alpha}$ as self-adjoint extensions of the densely defined closed symmetric operator 
\begin{equation*}
 Sf=-\Delta f\quad\dom(S)=\bigl\{f\in H^2(\dR^n) \, \big| \, f\!\upharpoonright_\cC=0\bigr\},
\end{equation*}
and make use of the
concept of so-called quasi boundary triples and their Weyl fucntions (see \cite{BL07,BL12}). 
It will then turn out in Theorem~\ref{dddthm2} that the trace class condition \eqref{resabmmm} 
is satisfied, and in the special case $\alpha(x)<0$, $x\in\cC$, the function 
\begin{align*}
& \xi(\lambda)    \\
& \quad =\sum_{j \in J}\lim_{\varepsilon\downarrow 0}\frac{1}{\pi} \Bigl(\big(\Im \big(\!\log \big(
\bigl(\overline{\cD_{\rm i}(\lambda+i\varepsilon)+\cD_{\rm e}(\lambda+i\varepsilon)}\bigr)^{-1} 
- \alpha^{-1}\big)\big)\big)\varphi_j,\varphi_j\Bigr)_{L^2(\cC)} 
\end{align*}
for a.e. $\lambda\in\dR$,
is a spectral shift function for the pair $\{H,H_{\delta,\alpha}\}$ such that $\xi(\lambda)=0$ for $\lambda < 0$; here $(\varphi_j)_{j \in J}$ 
is an orthonormal basis in $L^2(\cC)$.
For the case that no sign condition on the function $\alpha$ is assumed, a slightly more involved formula for the spectral shift function is provided 
in Theorem~\ref{dddthm} and in Corollary~\ref{mainthmcorchen3}.

Next, we briefly summarize the basic notation used in this paper: Let $\cG$, 
$\sH$, etc., be separable complex Hilbert spaces, $(\cdot,\cdot)_{\sH}$ the
scalar product in $\sH$ (linear in the first factor), and $I_{\sH}$ the identity operator
in $\sH$. If $T$ is a linear operator mapping (a subspace of\,) a
Hilbert space into another, $\dom(T)$ denotes the domain and $\ran(T)$ is the range 
of $T$. The closure
of a closable operator $S$ is denoted by $\ol S$. The spectrum and
resolvent set of a closed linear operator in $\sH$ will be denoted by
$\sigma(\cdot)$  and $\rho(\cdot)$, respectively.
The Banach spaces of bounded linear operators in $\sH$ are
denoted by $\cL(\sH)$; in the context of two
Hilbert spaces, $\sH_j$, $j=1,2$, we use the analogous abbreviation
$\cL(\sH_1, \sH_2)$. The $p$-th Schatten-von Neumann ideal consists of compact operators with singular values in $\ell^p$, $p>0$, 
and is denoted by $\sS_p(\sH)$ and $\sS_p(\sH_1,\sH_2)$.
For $\Omega \subseteq \bbR^n$ nonempty, $n \in \bbN$,  we suppress the 
$n$-dimensional Lebesgue measure $d^n x$ and use the shorthand notation 
$L^2(\Omega) := L^2(\Omega; d^n x)$; similarly, if $\partial \Omega$ is sufficiently regular we 
write $L^2(\partial \Omega) := L^2(\partial \Omega; d^{n-1} \sigma)$, with $d^{n-1} \sigma$ 
the surface measure on $\partial \Omega$. 
We also abbreviate $\bbC_{\pm} := \{z \in \bbC \, | \, \Im(z) \gtrless 0\}$ and 
$\bbN_0 = \bbN \cup \{0\}$.

\section{Quasi boundary triples and their Weyl functions}\label{section2}

In this preliminary section we briefly recall the concept of quasi boundary triples and their Weyl functions from
extension theory of symmetric operators, which will be used in the next sections.
We refer to \cite{BL07,BL12} for more details on quasi boundary triples and to \cite{BGP08,DM91,DM95,GG91,S12} for the closely related 
concepts of generalized and ordinary 
boundary triples. 

Throughout this section let $\sH$ be a separable Hilbert space and let $S$ be a densely defined closed symmetric operator in $\sH$. 

\begin{definition}\label{qbtdefinition}
Let $T\subset S^*$ be a linear operator in $\sH$ such that $\overline T=S^*$.
A triple $\{\cG,\Gamma_0,\Gamma_1\}$ is said to be a {\em quasi boundary triple}
for $T\subset S^*$ if $\cG$ is a Hilbert space and $\Gamma_0,\Gamma_1:\dom (T)\rightarrow\cG$
are linear mappings such that the following conditions $(i)$--$(iii)$ are satisfied:
\begin{itemize}
  \item [$(i)$] The abstract Green's identity
    \begin{equation*}
      (Tf,g)_\sH-(f,Tg)_\sH=(\Gamma_1 f,\Gamma_0 g)_\cG-(\Gamma_0 f,\Gamma_1 g)_\cG
    \end{equation*}
    holds for all $f,g\in\dom (T)$. 
  \item [$(ii)$] The range of the map $(\Gamma_0,\Gamma_1)^\top:\dom(T)\rightarrow\cG\times\cG$ is dense. 
  \item [$(iii)$] The operator $A_0:=T\upharpoonright\ker(\Gamma_0)$ is self-adjoint in $\sH$.
\end{itemize}
\end{definition}

The next theorem from \cite{BL07,BL12} 
contains a sufficient condition for a triple $\{\cG,\Gamma_0,\Gamma_1\}$ to be a quasi boundary triple. It will be used in the proof of Theorem~\ref{dddthm}.

\begin{theorem}\label{ratemal}
Let $\sH$ and $\cG$ be separable Hilbert spaces and let $T$ be a linear operator in $\sH$.
Assume that $\Gamma_0,\Gamma_1: \dom (T)\rightarrow\cG$ are linear mappings such
that the following conditions $(i)$--$(iii)$ hold:
\begin{itemize}
\item[$(i)$] The abstract Green's identity
\begin{equation*}
 (Tf,g)_\sH-(f,Tg)_\sH=(\Gamma_1 f,\Gamma_0 g)_\cG-(\Gamma_0 f,\Gamma_1 g)_\cG
\end{equation*}
holds for all $f,g\in\dom(T)$. 
\item [$(ii)$]
The range of 
$(\Gamma_0,\Gamma_1)^\top: \dom (T)\rightarrow\cG\times\cG$
is dense and $\ker(\Gamma_0)\cap\ker(\Gamma_1)$ is dense in $\sH$. 
\item[$(iii)$] $T\upharpoonright \ker(\Gamma_0)$ is an extension of a self-adjoint
operator $A_0$.
\end{itemize}
Then $$S:= T\upharpoonright\bigl(\ker(\Gamma_0)\cap\ker(\Gamma_1)\bigr)$$ 
is a densely defined closed
symmetric operator in $\sH$ such that $\overline T= S^*$ holds and $\{\cG,\Gamma_0,\Gamma_1\}$ is a
quasi boundary triple for $S^*$ with $A_0=T\upharpoonright \ker(\Gamma_0)$.
\end{theorem}

Next, we recall the definition of the $\gamma$-field $\gamma$ and Weyl function $M$ associated to a quasi boundary triple, which
is formally the same as in \cite{DM91,DM95} for the case of ordinary or generalized boundary triples. 
For this let $\{\cG,\Gamma_0,\Gamma_1\}$ be a quasi boundary triple for $T\subset S^*$ with $A_0=T\upharpoonright\ker(\Gamma_0)$.
We note that the direct sum decomposition
\begin{equation*}
  \dom (T) = \dom (A_0)\,\dot +\,\ker(T - z I_{\sH})
  = \ker(\Gamma_0)\,\dot +\,\ker(T- z I_{\sH})
\end{equation*}
of $\dom(T)$ holds for all $z \in \rho(A_0)$, and hence the mapping 
$\Gamma_0\upharpoonright \ker(T - z I_{\sH})$
is injective for all $z \in \rho(A_0)$ and its range coincides with $\ran(\Gamma_0)$.

\begin{definition}\label{gwdeffi}
Let $T\subset S^*$ be a linear operator in $\sH$ such that $\overline T=S^*$
and let $\{\cG,\Gamma_0,\Gamma_1\}$ be a quasi boundary triple for $T\subset S^*$ with $A_0=T\upharpoonright\ker(\Gamma_0)$.
The {\em $\gamma$-field} $\gamma$ and the {\em Weyl function} $M$ corresponding to $\{\cG,\Gamma_0,\Gamma_1\}$ are operator-valued functions on $\rho(A_0)$ which are defined by
\begin{equation*}
   z \mapsto\gamma(z):=\bigl(\Gamma_0\!\upharpoonright\ker(T - z I_{\sH})\bigr)^{-1}  \text{ and }   z \mapsto  
  M(z) := \Gamma_1\bigl(\Gamma_0\!\upharpoonright\ker(T - z I_{\sH})\bigr)^{-1}.
\end{equation*}
\end{definition}

Various useful properties of the $\gamma$-field and Weyl function associated to a quasi boundary triple 
were provided in \cite{BL07,BL12,BLL13-3}, see also \cite{BGP08,DM91,DM95,S12} for the special cases of ordinary and 
generalized boundary triples. In the following we only recall some properties important for our purposes. We first note that 
the values $\gamma(z)$, $z \in \rho(A_0)$, 
of the $\gamma$-field are operators defined on 
the dense subspace $\ran(\Gamma_0)\subset\cG$ which map onto
$\ker(T - z I_{\sH})\subset\sH$. The operators $\gamma(z)$, $z \in \rho(A_0)$, are bounded and admit continuous extensions $\overline{\gamma(z)}\in\cL(\cG,\sH)$,
the function $z\mapsto \overline{\gamma(z)}$ is analytic on $\rho(A_0)$, and one has
\begin{equation*}
 \frac{d^k}{dz^k}\overline{\gamma(z)}=k! \, (A_0 - z I_{\sH})^{-k}\overline{\gamma(z)},\qquad k \in \bbN_0,\,\,z \in \rho(A_0).
\end{equation*}
For the adjoint operators $\gamma(z)^*\in\cL(\sH,\cG)$, $z \in \rho(A_0)$, it follows from the abstract Green's identity in Definition~\ref{qbtdefinition}\,$(i)$ that  
\begin{equation}\label{gstar}
 \gamma(z)^*=\Gamma_1(A_0-{\ol z} I_{\sH})^{-1},\quad z \in \rho(A_0),
\end{equation}
and one has 
\begin{equation}\label{gammad2}
 \frac{d^k}{dz^k}\gamma({\ol z})^*=k! \, \gamma({\ol z})^*(A_0 - z I_{\sH})^{-k},\qquad k \in \bbN_0,\,\,z \in \rho(A_0).
\end{equation}

The values $M(z)$, $z \in \rho(A_0)$, of the Weyl
function $M$ associated to a quasi boundary triple are operators in $\cG$ with $\dom (M(z)) = \ran(\Gamma_0)$ and $\ran (M(z)) \subseteq \ran (\Gamma_1)$
for all $z \in \rho(A_0)$. In general, $M(z)$ may be an unbounded operator, which is not necessarily closed, but closable.
One can show that $ z \mapsto M(z)\varphi$ is holomorphic on $\rho(A_0)$ for all $\varphi\in\ran(\Gamma_0)$ and in the case, 
where the values $M(z)$ are densely defined bounded operators for some,
and hence for all $z \in \rho(A_0)$, one has 
\begin{equation}\label{gammad3}
 \frac{d^k}{dz^k}\overline{M(z)}=k! \, \gamma({\ol z})^*(A_0 - z I_{\sH})^{-(k-1)}\overline{\gamma(z)}, 
  \quad k \in \bbN, \; z \in \rho(A_0). 
\end{equation}

\section{A representation formula for  the spectral shift function}\label{ssfsec}

Let $A$ and $B$ be self-adjoint operators in a separable Hilbert space $\sH$ and assume that the closed symmetric operator $S=A\cap B$, that is,
\begin{equation}\label{jass}
 Sf=Af=Bf,\quad\dom(S) = \bigl\{f\in\dom(A)\cap\dom(B) \, | \, Af=Bf\bigr\},
\end{equation}
is densely defined.  
According to \cite[Proposition~2.4]{BGN17} there exists a quasi boundary triple $\{\cG,\Gamma_0,\Gamma_1\}$ 
with $\gamma$-field $\gamma$ and Weyl function $M$ such that
\begin{equation}\label{hoho2}
 A=T\upharpoonright\ker(\Gamma_0)\, \text{ and } \, B=T\upharpoonright\ker(\Gamma_1),
\end{equation}
and
\begin{equation}\label{resab}
 (B - z I_{\sH})^{-1}-(A - z I_{\sH})^{-1}=-\gamma(z) M(z)^{-1}\gamma({\ol z})^*,\quad z \in \rho(A)\cap\rho(B).
\end{equation}

Next we recall the main result in the abstract part of \cite{BGN17}, in which  
an explicit expression for a spectral shift function of the pair $\{A,B\}$ in terms of the 
Weyl function $M$ is found. We refer the reader to \cite[Section~4]{BGN17} for a detailed discussion and the proof of Theorem~\ref{mainssf2}.
We shall use the logarithm of a boundedly invertible dissipative operator in the formula for the spectral shift function below. Here we define for $K\in\cL(\cG)$ with 
$\Im(K) \geq 0$ and $0\subset\rho(K)$ the logarithm as
\begin{equation*}
 \log (K):=-i\int_0^\infty \bigl[(K+ i \lambda I_{\cG})^{-1}-(1+ i \lambda)^{-1}I_\cG\bigr] \, d\lambda;
\end{equation*}
cf. \cite[Section~2]{GMN99} for more details. We only mention that
$\log (K)\in\cL(\cG)$ by \cite[Lemma~2.6]{GMN99}.

\begin{theorem}\label{mainssf2}
Let $A$ and $B$ be self-adjoint operators in a separable Hilbert space $\sH$ and assume that for some $\zeta_0 \in\rho(A)\cap\rho(B)\cap\dR$ the sign condition
 \begin{equation}\label{sign333}
  (A-\zeta_0 I_{\sH})^{-1}\geq (B- \zeta_0 I_{\sH})^{-1}
 \end{equation}
 holds. Let the closed symmetric operator $S=A\cap B$ in \eqref{jass} be densely defined and 
 let $\{\cG,\Gamma_0,\Gamma_1\}$ be a quasi boundary triple with $\gamma$-field $\gamma$ and Weyl function $M$ 
such that
\eqref{hoho2}, and hence also \eqref{resab}, hold. Assume that $M(z_1)$, $M(z_2)^{-1}$ are 
bounded $($not necessarily everywhere defined\,$)$ 
operators in $\cG$ for some $z_1, z_2 \in \rho(A)\cap\rho(B)$
and that for some $k \in \bbN_0$, all $p,q \in \bbN_0$, and all $z \in \rho(A)\cap\rho(B)$, 
\begin{equation*}
 \left(\frac{d^p}{d z^p}\overline{\gamma(z)}\right)\frac{d^q}{dz^q}\bigl( M(z)^{-1} \gamma({\ol z})^*\bigr)\in\sS_1(\sH),\quad p+q=2k,
\end{equation*}
\begin{equation*}
 \left(\frac{d^q}{dz^q}\bigl( M(z)^{-1} \gamma({\ol z})^*\bigr)\right)\frac{d^p}{dz^p}\overline{\gamma(z)}\in\sS_1(\cG),\quad p+q=2k,
\end{equation*}
and 
 \begin{equation*}
   \frac{d^j}{dz^j} \overline{M (z)}\in\sS_{(2k+1)/j}(\cG),\quad j=1,\dots,2k+1.
 \end{equation*}
Then the following assertions $(i)$ and $(ii)$ hold:
 \begin{itemize}
  \item [$(i)$] The difference of the $2k+1$th-powers of the resolvents of $A$ and $B$ is 
a trace class operator, that is,
\begin{equation*}
 \big[(B - z I_{\sH})^{-(2k+1)}-(A - z I_{\sH})^{-(2k+1)}\big] \in \sS_1(\sH)
\end{equation*}
holds for all $z\in\rho(A)\cap\rho(B)$. 
 \item [$(ii)$] For any orthonormal basis $\{\varphi_j\}_{j \in J}$ in $\cG$ the function 
  \begin{equation*}
   \xi(\lambda)=
   \sum_{j \in J} \lim_{\varepsilon\downarrow 0}\frac{1}{\pi}\bigl(\Im\big(\log\big(\overline{M(\lambda+i\varepsilon)}\big)\big)\varphi_j,\varphi_j\bigr)_\cG
   \, \text{ for a.e.~$\lambda \in \dR$}, 
  \end{equation*}
is a spectral shift function for the pair $\{A,B\}$ such that $\xi(\lambda)=0$ in an open neighborhood 
of $\zeta_0$; the function $\xi$ does not depend on the choice of the orthonormal basis $(\varphi_j)_{j \in J}$. In particular, the trace formula
\begin{equation*}
\begin{split}
& \tr_{\sH}\bigl( (B - z I_{\sH})^{-(2k+1)}-(A - z I_{\sH})^{-(2k+1)}\bigr)\\ 
 & \quad = - (2k+1) \int_\dR \frac{\xi(\lambda)\, d\lambda}{(\lambda - z)^{2k+2}}, 
 \quad z \in \rho(A)\cap\rho(B), 
\end{split} 
\end{equation*}
holds.
\end{itemize}
\end{theorem}

In the special case $k=0$ Theorem~\ref{mainssf2} can be reformulated and slightly improved; cf. \cite[Corollary~4.2]{BGN17}. Here the essential feature is 
that the limit $\Im(\log(\overline {M(\lambda+i0)}))$ exists in $\sS_1(\cG)$ 
for a.e. $\lambda\in\dR$.

\begin{corollary}\label{mainthmcorchen}
Let $A$ and $B$ be self-adjoint operators in a separable Hilbert space $\sH$ and assume that for some $\zeta_0 \in \rho(A)\cap\rho(B)\cap\dR$ the sign condition
 \begin{equation*}
  (A-\zeta_0 I_{\sH})^{-1}\geq (B-\zeta_0 I_{\sH})^{-1}
 \end{equation*}
 holds. Assume that the closed symmetric operator $S=A\cap B$ in \eqref{jass} is densely defined and 
 let $\{\cG,\Gamma_0,\Gamma_1\}$ be a quasi boundary triple with $\gamma$-field $\gamma$ and Weyl function $M$ such that \eqref{hoho2}, and hence also \eqref{resab}, hold. Assume that $M(z_1)$, 
$M(z_2)^{-1}$ are bounded $($not necessarily everywhere defined\,$)$ 
operators in $\cG$ for some $z_1,z_2\in\rho(A)$ and that $\overline{\gamma(z_0)}\in\sS_2(\cG,\sH)$ for some $z_0\in\rho(A)$. 
Then the following assertions $(i)$--$(iii)$ hold: 
\begin{itemize}
  \item [$(i)$]
The difference of the resolvents of $A$ and $B$ is 
a trace class operator, that is,
\begin{equation*}
 \big[(B - z I_{\sH})^{-1}-(A - z I_{\sH})^{-1}\big] \in \sS_1(\sH)
\end{equation*}
holds for all $z\in\rho(A)\cap\rho(B)$. 
  \item [$(ii)$] $\Im(\log (\overline{M(z)}))\in\sS_1(\cG)$ for all $z\in\dC\backslash\dR$ and the limit 
  $$\Im\big(\log\big(\overline{M(\lambda+i 0)}\big)\big):=\lim_{\varepsilon\downarrow 0}\Im\big(\log\big(\overline{M(\lambda+i\varepsilon)}\big)\big)$$ 
  exists for a.e.~$\lambda \in \dR$ in $\sS_1(\cG)$. 
  \item [$(iii)$] The function
  \begin{equation*}
   \xi(\lambda)=\frac{1}{\pi} \tr_{\cG}\bigl(\Im\big(\log\big(\overline{M(\lambda + i0)}\big)\big)\bigr) \, \text{ for a.e.~$\lambda \in \dR$}, 
  \end{equation*}
is a spectral shift function for the pair $\{A,B\}$ such that $\xi(\lambda)=0$ in an open neighborhood 
of $\zeta_0$ and the trace formula
\begin{equation*}
 \tr_{\sH}\bigl( (B - z I_{\sH})^{-1}-(A - z I_{\sH})^{-1}\bigr) 
 = -  \int_\dR \frac{\xi(\lambda)\, d\lambda}{(\lambda - z)^{2}} 
\end{equation*}
is valid for all $z \in \rho(A)\cap\rho(B)$.
\end{itemize}
\end{corollary}

We also recall from \cite[Section~4]{BGN17} how 
the sign condition \eqref{sign333} in the assumptions in Theorem~\ref{mainssf2} can be replaced by some weaker comparability condition,
which is satisfied in our main application in the next section.
Again, let $A$ and $B$ be self-adjoint operators in a separable Hilbert space $\sH$ and assume that there exists a self-adjoint operator $C$ in $\sH$ such that
\begin{equation*}
 (C-\zeta_A I_{\sH})^{-1}\geq (A-\zeta_A I_{\sH})^{-1}\, \text{ and } \,  
 (C-\zeta_B I_{\sH})^{-1}\geq (B-\zeta_B I_{\sH})^{-1}
\end{equation*}
for some $\zeta_A\in\rho(A)\cap\rho(C)\cap\dR$ and some $\zeta_B\in\rho(B)\cap\rho(C)\cap\dR$, respectively. Assume that the closed symmetric
operators $S_A=A\cap C$ and $S_B=B\cap C$ are both densely defined and choose quasi boundary triples $\{\cG_A,\Gamma_0^A,\Gamma_1^A\}$ and  
$\{\cG_B,\Gamma_0^B,\Gamma_1^B\}$ with $\gamma$-fields $\gamma_A,\gamma_B$ and Weyl functions $M_A$, $M_B$ for
\begin{equation*}
 T_A=S_A^*\upharpoonright\bigl(\dom (A)+\dom(C)\bigr)\, \text{ and } \, T_B=S_B^*\upharpoonright\bigl(\dom (B)+\dom(C)\bigr)
\end{equation*}
such
that
\begin{equation*}
 C=T_A\upharpoonright\ker(\Gamma_0^A)=T_B\upharpoonright\ker(\Gamma_0^B),
\end{equation*}
and
\begin{equation*}
  A=T_A\upharpoonright\ker(\Gamma_1^A)\, \text{ and } \, B=T_B\upharpoonright\ker(\Gamma_1^B),
\end{equation*}
(cf.\ \cite[Proposition~2.4]{BGN17}). Next, assume that for some $k \in \bbN_0$, 
the conditions in Theorem~\ref{mainssf2} are satisfied for the $\gamma$-fields $\gamma_A,\gamma_B$ and the Weyl functions $M_A$, $M_B$.
Then the difference of the $2k+1$-th powers of the resolvents of $A$ and $C$, and the difference of the $2k+1$-th powers of the resolvents of $B$ and $C$ are trace class operators,
and for orthonormal bases $(\varphi_j)_{j \in J}$ in $\cG_A$ and $(\psi_{\ell})_{\ell \in L}$ in $\cG_B$ ($J, L \subseteq \bbN$ appropriate index sets), 
\begin{equation*}
   \xi_A(\lambda)=\sum_{j \in J}\lim_{\varepsilon\downarrow 0}
   \frac{1}{\pi}\bigl(\Im\big(\log\big(\overline{M_A(\lambda+i\varepsilon)}\big)\big)\varphi_j,\varphi_j\bigr)_{\cG_A}\, \text{ for a.e.~$\lambda \in \dR$,}
  \end{equation*}
and 
\begin{equation*}
   \xi_B(\lambda)=\sum_{\ell \in L}\lim_{\varepsilon\downarrow 0}
   \frac{1}{\pi}\bigl(\Im\big(\log\big(\overline{M_B(\lambda+i\varepsilon)}\big)\big)
   \psi_{\ell},\psi_{\ell}\bigr)_{\cG_B}\, \text{ for a.e.~$\lambda \in \dR$,}
  \end{equation*}
are spectral shift functions for the pairs $\{C,A\}$ and $\{C,B\}$, respectively. It follows for $z \in \rho(A)\cap\rho(B)\cap\rho(C)$ 
that
\begin{equation*}
\begin{split}
 &\tr_{\sH}\bigl( (B - z I_{\sH})^{-(2k+1)}-(A - z I_{\sH})^{-(2k+1)}\bigr)\\
 &\quad= \tr_{\sH}\bigl( (B - z I_{\sH})^{-(2k+1)}-(C - z I_{\sH})^{-(2k+1)}\bigr)  
 \\ 
 & \qquad -\tr_{\sH}\bigl( (A - z I_{\sH})^{-(2k+1)}-(C - z I_{\sH})^{-(2k+1)}\bigr) \\
 &\quad = - (2k+1) \int_\dR \frac{[\xi_B(\lambda) - \xi_A(\lambda)] \, d\lambda }{(\lambda - z)^{2k+2}}
 \end{split}
 \end{equation*}
 and
 \begin{equation*}
  \int_\dR \frac{\vert \xi_B(\lambda)-\xi_A(\lambda)\vert \, d\lambda}{(1+\vert \lambda \vert)^{2m+2}} < \infty.
 \end{equation*}
  Therefore, 
\begin{equation} 
    \xi(\lambda) = \xi_B(\lambda)-\xi_A(\lambda)  \, \text{ for a.e.~$\lambda \in \dR$,}   \label{ssfab}
\end{equation}
 is a spectral shift function for the pair $\{A,B\}$, and 
  in the special case where $\cG_A = \cG_B := \cG$ and $(\varphi_j)_{j \in J}$ is an orthonormal basis in $\cG$, one infers that 
  \begin{equation}\label{ssfabc}
   \xi(\lambda) 
   =\sum_{j \in J}\lim_{\varepsilon\downarrow 0}\frac{1}{\pi} \Bigl(\bigl(\Im\bigl( \log \big(\overline{M_B(\lambda+i\varepsilon)}\big) 
   - \log \big(\overline{M_A(\lambda + i \varepsilon)}\big)\bigr)
   \varphi_j,\varphi_j\Bigr)_\cG 
   \end{equation}
  for a.e. $\lambda\in\dR$.
We emphasize that in contrast to the spectral shift function in Theorem~\ref{mainssf2}, here the spectral shift 
function $\xi$ in \eqref{ssfab} and \eqref{ssfabc} is not necessarily nonnegative.

\section{Schr\"{o}dinger operators with $\delta$-potentials supported on hypersurfaces}\label{ap2sec}

The aim of this section is to determine a spectral shift function for the pair $\{H,H_{\delta,\alpha}\}$, where $H=-\Delta$ is
the usual self-adjoint Laplacian in $L^2(\bbR^n)$, and $H_{\delta,\alpha}=-\Delta-\alpha \delta_\cC$ is a self-adjoint Schr\"{o}dinger operator with $\delta$-potential 
of strength $\alpha$ supported on a  compact hypersurface $\cC$ in $\dR^n$ which splits $\dR^n$ in a bounded interior domain and an unbounded exterior domain. 
Throughout this section we shall assume that the following hypothesis holds.

\begin{hypothesis}\label{hypo6}
Let $n \in \bbN$, $n \geq 2$, and $\Omega_{\rm i}$ be a nonempty, open, bounded interior domain 
in $\dR^n$ with a smooth boundary $\partial\Omega_{\rm i}$ and let 
$\Omega_{\rm e}=\dR^n\backslash\overline{\Omega_{\rm i}}$ be the corresponding 
exterior domain. The common boundary of the interior domain $\Omega_{\rm i}$ and exterior domain $\Omega_{\rm e}$ will be denoted by $\cC=\partial\Omega_{\rm e}=\partial\Omega_{\rm i}$. Furthermore, 
let $\alpha\in C^1(\cC)$ be a real-valued function on the boundary $\cC$.
\end{hypothesis}

We consider the self-adjoint operators in $L^2(\bbR^n)$, 
\begin{equation*}
 H f=-\Delta f,\quad \dom (H)=H^2(\dR^n),
\end{equation*}
and 
\begin{equation*}
\begin{split}
 H_{\delta,\alpha} f &= - \Delta f,\\
 \dom (H_{\delta,\alpha})&=\left\{f=\begin{pmatrix}f_{\rm i}\\[1mm] f_{\rm e}\end{pmatrix}\in H^2(\Omega_{\rm i})\times H^2(\Omega_{\rm e}) \, \Bigg| \, 
 \begin{matrix}\gamma_D^{\rm i}f_{\rm i}=\gamma_D^{\rm e}f_{\rm e},\quad\quad\!\\[1mm]
\alpha\gamma_D^{\rm i} f_{\rm i}=\gamma_N^{\rm i} f_{\rm i}+\gamma_N^{\rm e}f_{\rm e}\end{matrix}\right\},
\end{split}
 \end{equation*}
in  $L^2(\bbR^n)$.
Here $f_{\rm i}$ and $f_{\rm e}$ denote the restrictions of a function $f$ on $\dR^n$ onto $\Omega_{\rm i}$ and $\Omega_{\rm e}$, and 
$\gamma_D^{\rm i}$, $\gamma_D^{\rm e}$ and $\gamma_N^{\rm i}$, $\gamma_N^{\rm e}$ are the Dirichlet and Neumann trace operators on $H^2(\Omega_{\rm i})$ and 
$H^2(\Omega_{\rm e})$, respectively.  We note that 
$H_{\delta,\alpha}$  coincides with the self-adjoint operator associated
to the quadratic form
\begin{equation*}
 \mathfrak h_{\delta,\alpha}[f,g]=(\nabla f,\nabla g)-\int_\Sigma \alpha(x) f(x) \overline{g(x)}\,d\sigma(x),\quad f,g\in H^1(\dR^n), 
\end{equation*}
see \cite[Proposition~3.7]{BLL13-AHP} and \cite{BEKS94} for more details. For $c\in\dR$ we shall also make use of the self-adjoint operator
\begin{equation*}
\begin{split}
 H_{\delta,c} f &= - \Delta f,\\
 \dom (H_{\delta,c}) & = \left\{f=\begin{pmatrix}f_{\rm i}\\[1mm] f_{\rm e}\end{pmatrix}\in H^2(\Omega_{\rm i})\times H^2(\Omega_{\rm e}) \, \Bigg| \, 
 \begin{matrix}\gamma_D^{\rm i}f_{\rm i}=\gamma_D^{\rm e}f_{\rm e},\quad\quad\!\\[1mm]
c\gamma_D^{\rm i} f_{\rm i}=\gamma_N^{\rm i} f_{\rm i}+\gamma_N^{\rm e}f_{\rm e}\end{matrix}\right\}. 
\end{split}
\end{equation*}

The following lemma will be useful for 
the $\sS_p$-estimates in the proof of Theorem~\ref{dddthm}.

\begin{lemma}\label{usel}
Let 
$X\in \cL(L^2(\dR^n),H^t(\cC))$, and assume that 
$\ran(X)\subseteq H^s(\cC)$ for some $s>t\geq 0$. Then $X$ is compact 
and $($cf.\ \cite[Lemma 4.7]{BLL13-IEOT}$)$   
 \begin{equation*}
  X\in\sS_r\bigl(L^2(\dR^n),H^t(\cC)\bigr)\, \text{ for all } \, r> (n-1)/(s-t).
 \end{equation*}
\end{lemma}

Next we define interior and exterior
Dirichlet-to-Neumann maps $\cD_{\rm i}(z)$ and $\cD_{\rm e}(\zeta)$ as operators in $L^2(\cC)$ for all $z,\zeta \in \dC\backslash [0,\infty)=\rho(H)$. One notes 
that for $\varphi,\psi\in H^1(\cC)$ and $z,\zeta \in \dC\backslash [0,\infty)$, the boundary value problems 
\begin{equation}\label{i-bvp}
 -\Delta f_{{\rm i},z} = z f_{{\rm i},z},\quad \gamma_D^{\rm i} f_{{\rm i},z}=\varphi,
\end{equation}
and 
\begin{equation}\label{e-bvp}
 -\Delta f_{{\rm e},\zeta} = \zeta f_{{\rm e},\zeta},\quad \gamma_D^{\rm e} f_{{\rm e},\zeta} = \psi,
\end{equation}
admit unique solutions $f_{{\rm i},z}\in H^{3/2}(\Omega_{\rm i})$ and $f_{{\rm e},\zeta}\in H^{3/2}(\Omega_{\rm e})$, respectively. 
The corresponding solution operators are denoted by
\begin{equation*}
 P_{\rm i}(z):L^2(\cC) \rightarrow L^2(\Omega_{\rm i}),\quad \varphi\mapsto f_{{\rm i},z},
\end{equation*}
and
\begin{equation*}
 P_{\rm e}(\zeta):L^2(\cC) \rightarrow L^2(\Omega_{\rm e}),\quad \psi\mapsto f_{{\rm e},\zeta}.
\end{equation*}
The {\it interior Dirichlet-to-Neumann map} in $L^2(\cC)$, 
\begin{equation}\label{di}
 \cD_{\rm i}(z):H^1(\cC) \rightarrow L^2(\cC),\quad \varphi\mapsto \gamma_N^{\rm i} P_{\rm i}(z)\varphi,
\end{equation}
maps Dirichlet boundary values $\gamma_D^{\rm i} f_{{\rm i},z}$ of the solutions 
$f_{{\rm i},z}\in H^{3/2}(\Omega_{\rm i})$ of \eqref{i-bvp} onto the corresponding Neumann
boundary values $\gamma_N^{\rm i} f_{{\rm i},z}$, and the {\it exterior Dirichlet-to-Neumann map}  in $L^2(\cC)$, 
\begin{equation}\label{de}
 \cD_{\rm e}(\zeta):H^1(\cC) \rightarrow L^2(\cC),\quad 
 \psi\mapsto \gamma_N^{\rm e} P_{\rm e}(\zeta)\psi,
\end{equation}
maps Dirichlet boundary values $\gamma_D^{\rm e} f_{{\rm e},\zeta}$ of the solutions 
$f_{{\rm e},\zeta}\in H^{3/2}(\Omega_{\rm e})$ of \eqref{e-bvp} onto the corresponding Neumann
boundary values $\gamma_N^{\rm e} f_{{\rm e},\zeta}$. The interior and 
exterior Dirichlet-to-Neumann maps are both closed unbounded operators in $L^2(\cC)$.

In the next theorem a spectral shift function for the pair $\{H,H_{\delta,\alpha}\}$ is expressed in terms of the limits of the sum of the 
interior and exterior Dirichlet-to-Neumann map $\cD_{\rm i}(z)$ and $\cD_{\rm e}(z)$ and the function $\alpha$. It will turn out that the
operators $\cD_{\rm i}(z) +\cD_{\rm e}(z)$ are boundedly invertible for all $z\in\dC\backslash [0,\infty)$ and for our purposes
it is convenient to work with the function 
\begin{equation}\label{ee}
 z \mapsto\cE(z)=\bigl(\cD_{\rm i}(z) +\cD_{\rm e}(z)\bigr)^{-1},\quad z\in\dC\backslash[0,\infty). 
\end{equation}
It was shown in \cite[Proposition~3.2~(iii) and Remark~3.3]{BLL13-AHP} that 
$\cE(z)$ is a compact operator in $L^2(\cC)$ which extends the acoustic single layer potential for the Helmholtz equation, that is, 
\begin{equation*}
 (\cE(z)\varphi)(x)=\int_\cC G(z,x,y)\varphi(y)d\sigma(y),\quad x\in\cC,\,\,\varphi\in C^\infty(\cC),
\end{equation*}
where $G(z,\,\cdot\,,\,\cdot\,)$, $z\in\dC\backslash [0,\infty)$, represents the integral kernel of the resolvent of $H$ (cf.\ \cite[Chapter~6]{McL00} and \cite[Remark~3.3]{BLL13-AHP}). Explicitly, 
\begin{align*}
G(z,x,y) = (i/4) \big(2\pi z^{-1/2} |x - y|\big)^{(2-n)/2} 
H^{(1)}_{(n-2)/2}\big(z^{1/2}|x - y|\big),&   \\
z \in \bbC \backslash [0,\infty), \; \Im\big(z^{1/2}\big) > 0, 
\; x, y \in\bbR^n, \; x \neq y, \; n\ge 2.&
\end{align*}
Here $H^{(1)}_{\nu}(\, \cdot \,)$ denotes the Hankel function of the first kind 
with index $\nu\geq 0$ (cf.\ \cite[Sect.~9.1]{AS72}). 

We mention that the trace class property of the difference of the
$2k+1$th powers of the resolvents in the next theorem is known from \cite{BLL13-AHP} (see also \cite{BLL-Exner}).

\begin{theorem}\label{dddthm}
Assume Hypothesis~\ref{hypo6},  
let $\cE(z)$ be defined as in \eqref{ee}, let $\alpha\in C^1(\cC)$ be a real-valued function and fix 
$c>0$ such that $\alpha(x)<c$ for all $x\in\cC$. 
Then the following assertions $(i)$ and $(ii)$ hold for $k \in \bbN_0$ such that $k\geq (n-3)/4$:
 \begin{itemize}
  \item [$(i)$] The difference of the $2k+1$th-powers of the resolvents of $H$ and $H_{\delta,\alpha}$ is 
a trace class operator, that is,
\begin{equation*}
\big[(H_{\delta,\alpha} - z I_{L^2(\bbR^n)})^{-(2k+1)} 
- (H - z I_{L^2(\bbR^n)})^{-(2k+1)}\big] \in \sS_1\bigl(L^2(\dR^n)\bigr)
\end{equation*}
holds for all $z\in\rho(H_{\delta,\alpha})=\rho(H)\cap\rho(H_{\delta,\alpha})$. 
 \item [$(ii)$] For any orthonormal basis $(\varphi_j)_{j \in J}$ in $L^2(\cC)$ the function 
  \begin{equation*}
   \xi(\lambda) 
   =\sum_{j \in J} \lim_{\varepsilon\downarrow 0}\frac{1}{\pi} \Bigl(\bigl(\Im\bigl( \log (\cM_\alpha(\lambda+i\varepsilon)) - \log (\cM_0(\lambda + i \varepsilon))\bigr)\bigr)
   \varphi_j,\varphi_j\Bigr)_{L^2(\cC)}
  \end{equation*}
for a.e. $\lambda\in\dR$ with 
\begin{align}
\cM_0(z)&= - c^{-1}\bigl(c \cE(z) - I_{L^2(\cC)}\bigr)^{-1},   \lb{4.5a} \\
\cM_\alpha(z)&= (c-\alpha)^{-1}\bigl(\alpha\cE(z)- I_{L^2(\cC)}\bigr)\bigl(c \cE(z) - I_{L^2(\cC)}\bigr)^{-1},   \lb{4.5b} 
\end{align}
for $z\in\dC\backslash\dR$,
is a spectral shift function for the pair $\{H,H_{\delta,\alpha}\}$ such that $\xi(\lambda)=0$ for 
$\lambda < \inf(\sigma(H_{\delta,c}))$ and the trace formula
\begin{align*}
 \tr_{L^2(\bbR^n)}\bigl( (H_{\delta,\alpha} -  z I_{L^2(\bbR^n)})^{-(2k+1)} 
& - (H - z I_{L^2(\bbR^n)})^{-(2k+1)}\bigr)   \\ 
& \quad = - (2k+1) \int_\dR \frac{\xi(\lambda)\, d\lambda}{(\lambda - z)^{2k+2}} 
 \end{align*}
is valid for all $z \in \rho(H_{\delta,\alpha})=\rho(H)\cap\rho(H_{\delta,\alpha})$.
\end{itemize}
\end{theorem}

\begin{proof}
The structure and underlying idea of the proof of Theorem~\ref{dddthm} is as follows: In the first two steps
a suitable quasi boundary triple and its Weyl function are constructed. In the third step it is shown that 
the assumptions in Theorem~\ref{mainssf2} are satisfied. 

\noindent 
{\it Step 1.}
Since $c-\alpha(x)\not=0$ for all $x\in\cC$ by assumption, the closed symmetric operator $S=H_{\delta,c}\cap H_{\delta,\alpha}$ is given by
\begin{equation*}
Sf=-\Delta f, \quad \dom (S) = \bigl\{f\in H^2(\dR^n) \, \big| \, 
\gamma_D^{\rm i}f_{\rm i}=\gamma_D^{\rm e}f_{\rm e}=0\bigr\}. 
 \end{equation*}
In this step we show that the operator
\begin{equation*}
T = - \Delta,    \quad  
\dom (T) = \left\{f=\begin{pmatrix}f_{\rm i}\\f_{\rm e}\end{pmatrix}\in H^2(\Omega_{\rm i})\times H^2(\Omega_{\rm e}) \, \bigg| \, \gamma_D^{\rm i}f_{\rm i}=\gamma_D^{\rm e}f_{\rm e}\right\}, 
 \end{equation*}
satisfies $\overline T=S^*$ and that $\{L^2(\cC),\Gamma_0,\Gamma_1\}$, where
\begin{equation}\label{qbtw1}
 \Gamma_0 f=c\gamma_D^{\rm i} f_{\rm i}-(\gamma_N^{\rm i}f_{\rm i}+\gamma_N^{\rm e}f_{\rm e}),\quad \dom(\Gamma_0)=\dom(T),
\end{equation}
and
\begin{equation}\label{qbtw2}
 \Gamma_1 f= \frac{1}{c-\alpha}\bigl(\alpha\gamma_D^{\rm i} f_{\rm i}-(\gamma_N^{\rm i}f_{\rm i}+\gamma_N^{\rm e}f_{\rm e})\bigr),\quad 
 \dom(\Gamma_1)=\dom(T),
\end{equation} 
is a quasi boundary triple for $T\subset S^*$ such that
\begin{equation}\label{hhd}
 H_{\delta,c}=T\upharpoonright\ker(\Gamma_0)\, \text{ and } \, 
 H_{\delta,\alpha}=T\upharpoonright\ker(\Gamma_1).
\end{equation}

For the proof of this fact we make use of Theorem~\ref{ratemal} and verify next that 
assumptions $(i)$--$(iii)$  in Theorem~\ref{ratemal} are satisfied with the above choice 
of $S$, $T$ and boundary maps $\Gamma_0$ and $\Gamma_1$. For $f,g\in\dom (T)$ one computes 
\begin{equation*}
 \begin{split}
&(\Gamma_1 f,\Gamma_0 g)_{L^2(\cC)}-(\Gamma_0 f,\Gamma_1 g)_{L^2(\cC)}\\
&\quad =\Bigl(\tfrac{1}{c-\alpha}\bigl(\alpha\gamma_D^{\rm i} f_{\rm i}-(\gamma_N^{\rm i}f_{\rm i}+\gamma_N^{\rm e}f_{\rm e})\bigr), 
                c\gamma_D^{\rm i} g_{\rm i}-(\gamma_N^{\rm i}g_{\rm i}+\gamma_N^{\rm e}g_{\rm e})\Bigr)_{L^2(\cC)}\\
&\qquad - \Bigl(c\gamma_D^{\rm i} f_{\rm i}-(\gamma_N^{\rm i}f_{\rm i}+\gamma_N^{\rm e}f_{\rm e}),
                     \tfrac{1}{c-\alpha}\bigl(\alpha\gamma_D^{\rm i} g_{\rm i}-(\gamma_N^{\rm i}g_{\rm i}+\gamma_N^{\rm e}g_{\rm e})\bigr)\Bigr)_{L^2(\cC)}\\
&\quad =-\bigl(\tfrac{\alpha}{c-\alpha}\gamma_D^{\rm i} f_{\rm i}, \gamma_N^{\rm i}g_{\rm i}+\gamma_N^{\rm e}g_{\rm e}\bigr)_{L^2(\cC)}  
-\bigl(\gamma_N^{\rm i}f_{\rm i}+\gamma_N^{\rm e}f_{\rm e}, \tfrac{c}{c-\alpha}\gamma_D^{\rm i} g_{\rm i}\bigr)_{L^2(\cC)}\\
&\qquad +\bigl(\tfrac{c}{c-\alpha}\gamma_D^{\rm i} f_{\rm i}, \gamma_N^{\rm i}g_{\rm i}+\gamma_N^{\rm e}g_{\rm e}\bigr)_{L^2(\cC)} 
 +\bigl(\gamma_N^{\rm i}f_{\rm i}+\gamma_N^{\rm e}f_{\rm e}, \tfrac{\alpha}{c-\alpha}\gamma_D^{\rm i} g_{\rm i}\bigr)_{L^2(\cC)}\\                    
&\quad =\bigl(\gamma_D^{\rm i} f_{\rm i},\gamma_N^{\rm i}g_{\rm i}+\gamma_N^{\rm e}g_{\rm e}\bigr)_{L^2(\cC)} -\bigl(\gamma_N^{\rm i}f_{\rm i}+\gamma_N^{\rm e}f_{\rm e},\gamma_D^{\rm i} g_{\rm i}\bigr)_{L^2(\cC)},
 \end{split}
\end{equation*}
and on the other hand, Green's identity and $\gamma_D^{\rm i}f_{\rm i}=\gamma_D^{\rm e}f_{\rm e}$ and $\gamma_D^{\rm i}g_{\rm i}=\gamma_D^{\rm e}g_{\rm e}$ yield 
\begin{equation*}
 \begin{split}
&(Tf,g)_{L^2(\dR^n)}-(f,Tg)_{L^2(\dR^n)}\\
&\quad =(-\Delta f_{\rm i},g_{\rm i})_{L^2(\Omega_{\rm i})}-(f_{\rm i},-\Delta g_{\rm i})_{L^2(\Omega_{\rm i})}+(-\Delta f_{\rm e},g_{\rm e})_{L^2(\Omega_{\rm e})} \\
& \qquad 
-(f_{\rm e},-\Delta g_{\rm e})_{L^2(\Omega_{\rm e})}\\
&\quad =(\gamma_D^{\rm i}f_{\rm i},\gamma_N^{\rm i}g_{\rm i})_{L^2(\cC)}-(\gamma_N^{\rm i}f_{\rm i},\gamma_D^{\rm i}g_{\rm i})_{L^2(\cC)}\\
&\qquad
 +(\gamma_D^{\rm e}f_{\rm e},\gamma_N^{\rm e}g_{\rm e})_{L^2(\cC)}-(\gamma_N^{\rm e}f_{\rm e},\gamma_D^{\rm e}g_{\rm e})_{L^2(\cC)}\\
 &\,\,=\bigl(\gamma_D^{\rm i} f_{\rm i},\gamma_N^{\rm i}g_{\rm i}+\gamma_N^{\rm e}g_{\rm e}\bigr)_{L^2(\cC)} 
  -\bigl(\gamma_N^{\rm i}f_{\rm i}+\gamma_N^{\rm e}f_{\rm e},\gamma_D^{\rm i} g_{\rm i}\bigr)_{L^2(\cC)},
\end{split}
\end{equation*}
and hence condition $(i)$ in Theorem~\ref{ratemal} holds. Next, in order to show that $\ran (\Gamma_0,\Gamma_1)^\top$ is dense in $L^2(\cC)$ we recall that 
\begin{equation*}
 \begin{pmatrix}
  \gamma_D^{\rm i}\\[1mm] \gamma_N^{\rm i}
 \end{pmatrix}:H^2(\Omega_{\rm i})\rightarrow H^{3/2}(\cC)\times H^{1/2}(\cC)
\end{equation*}
and
\begin{equation*}
 \begin{pmatrix}
  \gamma_D^{\rm e}\\[1mm] \gamma_N^{\rm e}
 \end{pmatrix}:H^2(\Omega_{\rm e})\rightarrow H^{3/2}(\cC)\times H^{1/2}(\cC)
\end{equation*}
are surjective mappings. It follows that also the mapping
\begin{equation}\label{tmap}
 \begin{pmatrix}
  \gamma_D^{\rm i} \\[1mm] \gamma_N^{\rm i}+\gamma_N^{\rm e}
 \end{pmatrix}:\dom(T) \rightarrow H^{3/2}(\cC)\times H^{1/2}(\cC)
\end{equation}
is surjective, and since the $2\times 2$-block operator matrix
\begin{equation*}
 \Theta:=\begin{pmatrix} c I_{L^2(\cC)} & -I_{L^2(\cC)} \\[1mm]  \frac{\alpha}{c-\alpha} I_{L^2(\cC)}
 &  - \frac{1}{c-\alpha} I_{L^2(\cC)} \end{pmatrix} 
\end{equation*}
is an isomorphism in $L^2(\cC)\times L^2(\cC)$, it follows that the range of the mapping, 
\begin{equation*}
\begin{pmatrix}\Gamma_0 \\ \Gamma_1\end{pmatrix}= \Theta  \begin{pmatrix}
  \gamma_D^{\rm i} \\[1mm] \gamma_N^{\rm i}+\gamma_N^{\rm e}
 \end{pmatrix}  :\dom(T)\rightarrow  L^2(\cC)\times L^2(\cC), 
\end{equation*}
is dense. Furthermore, as $C_0^\infty(\Omega_{\rm i})\times C_0^\infty(\Omega_{\rm e})$ is contained in $\ker(\Gamma_0)\cap\ker(\Gamma_1)$, it is clear that 
$\ker(\Gamma_0)\cap\ker(\Gamma_1)$ is dense in $L^2(\dR^n)$. Hence one concludes that 
condition $(ii)$ in Theorem~\ref{ratemal} is satisfied. Condition $(iii)$ in Theorem~\ref{ratemal} is satisfied since \eqref{hhd} holds by construction
and $H_{\delta,c}$ is self-adjoint. Thus, Theorem~\ref{ratemal}
implies that the closed symmetric operator 
\begin{equation*}
 T\upharpoonright\bigl(\ker(\Gamma_0)\cap\ker(\Gamma_1)\bigr)=H_{\delta,c}\cap H_{\delta,\alpha}=S
\end{equation*}
is densely defined, its adjoint coincides with $\overline T$, and 
$\{L^2(\cC),\Gamma_0,\Gamma_1\}$ is a quasi boundary triple for $T\subset S^*$ such that \eqref{hhd} holds.

\vskip 0.2cm\noindent
{\it Step 2.} In this step we prove that for $z \in \rho(H_{\delta,c})\cap\rho(H)$ 
the Weyl function corresponding to the quasi boundary triple 
$\{L^2(\cC),\Gamma_0,\Gamma_1\}$ is given by 
\begin{equation}\label{mddf}
\begin{split}
& M(z)= \frac{1}{c-\alpha}\bigl(\alpha\cE_{1/2}(z)- I_{L^2(\cC)}\bigr)\bigl(c \cE_{1/2}(z) 
- I_{L^2(\cC)}\bigr)^{-1},\\
& \dom(M(z)) =H^{1/2}(\cC),
\end{split}
\end{equation}
where $\cE_{1/2}(z)$ denotes the restriction of the operator $\cE(z)$ in \eqref{ee} onto $H^{1/2}(\cC)$. Furthermore, we verify 
that $M(z_1)$ and $M(z_2)^{-1}$ are bounded for some $z_1,z_2\in\dC\backslash\dR$, and we conclude that the closures
of the operators $M(z)$, $z\in\dC\backslash\dR$, in $L^2(\cC)$ are given by the operators $\cM_\alpha(z)$ in \eqref{4.5a}, \eqref{4.5b}.

It will first be shown that the operator $\cE(z)$ and its restriction $\cE_{1/2}(z)$ are well-defined 
for all $z \in \rho(H)=\dC\backslash [0,\infty)$. 
For this fix $z\in\dC\backslash [0,\infty)$, and let
\begin{equation}\label{flt}
 f_z =\begin{pmatrix} f_{{\rm i},z} \\ f_{{\rm e},z}\end{pmatrix} \in H^{3/2}(\Omega_{\rm i})\times H^{3/2}(\Omega_{\rm e})
\end{equation}
such that $\gamma_D^{\rm i}f_{{\rm i},z}=\gamma_D^{\rm e}f_{{\rm e},z}$, and
\begin{equation*}
-\Delta f_{{\rm i},z} = z f_{{\rm i},z}\, \text{ and } \, 
-\Delta f_{{\rm e},z} = z f_{{\rm e},z}.  
\end{equation*}
From the definition of $\cD_{\rm i}(z)$ and $\cD_{\rm e}(z)$ in \eqref{di} and \eqref{de} one concludes 
that 
\begin{equation}\label{useit2}
\begin{split}
 \bigl(\cD_{\rm i}(z)+\cD_{\rm e}(z)\bigr)\gamma_D^{\rm i}f_{{\rm i},z}&=\cD_{\rm i}(z)\gamma_D^{\rm i}f_{{\rm i},z}
 +\cD_{\rm e}(z)\gamma_D^{\rm e}f_{{\rm e},z}\\
 &=\gamma_N^{\rm i}f_{{\rm i},z}+\gamma_N^{\rm e}f_{{\rm e},z}.
 \end{split}
 \end{equation}
This also proves that $\cD_{\rm i}(z)+\cD_{\rm e}(z)$ is injective for $z\in\dC\backslash [0,\infty)$.
In fact, otherwise there would exist a 
function $f_z =(f_{{\rm i},z}, f_{{\rm e},z})^\top\not=0$ as in \eqref{flt} which would satisfy both conditions
\begin{equation}\label{plkoplko}
 \gamma_D^{\rm i}f_{{\rm i},z}=\gamma_D^{\rm e}f_{{\rm e},z} \, 
 \text{ and } \,\gamma_N^{\rm i}f_{{\rm i},z}+\gamma_N^{\rm e}f_{{\rm e},z}=0,
\end{equation}
and hence for all $h\in\dom(H)=H^2(\dR^n)$,  
Green's identity together with the conditions \eqref{plkoplko} would imply 
\begin{equation}\label{fvb}
\begin{split}
 (H h,f_z &)_{L^2(\dR^n)}-(h, z f_z )_{L^2(\dR^n)}\\
&\quad =(-\Delta h_{\rm i},f_{{\rm i},z})_{L^2(\Omega_{\rm i})}-(h_{\rm i},-\Delta f_{{\rm i},z})_{L^2(\Omega_{\rm i})}\\
&\qquad\qquad
            + (-\Delta h_{\rm e},f_{{\rm e},z})_{L^2(\Omega_{\rm e})}-(h_{\rm e},-\Delta f_{{\rm e},z})_{L^2(\Omega_{\rm e})}\\
&\quad =(\gamma_D^{\rm i}h_{\rm i},\gamma_N^{\rm i}f_{{\rm i},z})_{L^2(\cC)} 
- (\gamma_N^{\rm i}h_{\rm i},\gamma_D^{\rm i}f_{{\rm i},z})_{L^2(\cC)}\\       
&\qquad\qquad  +(\gamma_D^{\rm e}h_{\rm e},\gamma_N^{\rm e}f_{{\rm e},z})_{L^2(\cC)} 
- (\gamma_N^{\rm e}h_{\rm e},\gamma_D^{\rm e}f_{{\rm e},z})_{L^2(\cC)}  \\
&\quad =0,
\end{split}
\end{equation}
that is, $f_z \in\dom(H)$ 
and $H f_z = z f_z $; a contradiction since $z \in \rho(H)$. Hence,  
$$\ker\bigl(\cD_{\rm i}(z)+\cD_{\rm e}(z)\bigr)=\{0\},\quad z\in\dC\backslash [0,\infty),$$
and if we denote the restrictions of $\cD_{\rm i}(z)$ and $\cD_{\rm e}(z)$ onto $H^{3/2}(\cC)$ by $\cD_{{\rm i},3/2}(z)$
and $\cD_{{\rm e},3/2}(z)$, respectively, then also $\ker(\cD_{{\rm i},3/2}(z)+\cD_{{\rm e},3/2}(z))=\{0\}$ for $z\in\dC\backslash [0,\infty)$. Thus, we have shown that
$\cE(z)$ and its restriction $\cE_{1/2}(z)$ 
are well-defined for all $z \in \rho(H)=\dC\backslash [0,\infty)$. 

Furthermore, if the function $f_z $ in \eqref{flt} belongs to 
$H^2(\Omega_{\rm i})\times H^2(\Omega_{\rm e})$, that is, $f_z \in\ker(T - z I_{L^2(\bbR^n)})$, then
$\gamma_D^{\rm i}f_{{\rm i},z}=\gamma_D^{\rm e}f_{{\rm e},z}\in H^{3/2}(\cC)$ and hence besides \eqref{useit2} one also has 
\begin{equation}\label{useit3}
 \bigl(\cD_{{\rm i},3/2}(z)+\cD_{{\rm e},3/2}(z)\bigr)\gamma_D^{\rm i}f_{{\rm i},z}=\gamma_N^{\rm i}f_{{\rm i},z}+\gamma_N^{\rm e}f_{{\rm e},z}\in H^{1/2}(\cC).
 \end{equation} 
One concludes from \eqref{useit3} that
\begin{equation*}
\cE_{1/2}(z)
\bigl(\gamma_N^{\rm i}f_{{\rm i},z}+\gamma_N^{\rm e}f_{{\rm e},z}\bigr)=\gamma_D^{\rm i}f_{{\rm i},z},
\end{equation*}
and from \eqref{qbtw1} one then obtains 
\begin{equation}\label{ggb}
\begin{split}
\bigl(c\cE_{1/2}(z)-I_{L^2(\cC)}\bigr)\bigl(\gamma_N^{\rm i}f_{{\rm i},z}+\gamma_N^{\rm e}f_{{\rm e},z}\bigr)
&=c\gamma_D^{\rm i}f_{{\rm i},z}-\bigl(\gamma_N^{\rm i}f_{{\rm i},z} 
+ \gamma_N^{\rm e}f_{{\rm e},z}\bigr)\\
&=\Gamma_0f_z,  
\end{split}
\end{equation}
and
\begin{equation}\label{ggb2}
\bigl(\alpha\cE_{1/2}(z)-I_{L^2(\cC)}\bigr)\bigl(\gamma_N^{\rm i}f_{{\rm i},z}+\gamma_N^{\rm e}f_{{\rm e},z}\bigr)
=\alpha\gamma_D^{\rm i}f_{{\rm i},z}-\bigl(\gamma_N^{\rm i}f_{{\rm i},z}+\gamma_N^{\rm e}f_{{\rm e},z}\bigr).
\end{equation}
For $z \in \rho(H_{\delta,c})\cap\rho(H)$ one verifies $\ker(c\cE_{1/2}(z)-I_{L^2(\cC)})=\{0\}$ with the help of \eqref{ggb}. Then 
\eqref{qbtw1} and \eqref{tmap} yield
\begin{equation*}
\ran\bigl(c\cE_{1/2}(z)-I_{L^2(\cC)}\bigr)=\ran(\Gamma_0)=H^{1/2}(\cC).
\end{equation*}
Thus, it follows from \eqref{ggb}, \eqref{ggb2}, and \eqref{qbtw2} that
\begin{equation*}
 \begin{split}
  & \frac{1}{c-\alpha}\bigl(\alpha\cE_{1/2}(z)- I_{L^2(\cC)}\bigr)\bigl(c \cE_{1/2}(z) - I_{L^2(\cC)}\bigr)^{-1}\Gamma_0f_z \\
  &\qquad = \frac{1}{c-\alpha}\bigl(\alpha\cE_{1/2}(z)- I_{L^2(\cC)}\bigr)\bigl(\gamma_N^{\rm i}f_{{\rm i},z} 
  + \gamma_N^{\rm e}f_{{\rm e},z}\bigr)\\
  &\qquad = \frac{1}{c-\alpha}\Bigl(\alpha\gamma_D^{\rm i}f_{{\rm i},z} 
  - \bigl(\gamma_N^{\rm i}f_{{\rm i},z}+\gamma_N^{\rm e}f_{{\rm e},z}\bigr)\Bigr)\\
  &\qquad =\Gamma_1 f_z 
 \end{split}
\end{equation*}
holds for all $z \in \rho(H_{\delta,c})\cap\rho(H)$. This proves that the Weyl function corresponding to the quasi boundary triple \eqref{qbtw1}--\eqref{qbtw2}
is given by \eqref{mddf}.

Next it will be shown that $M(z)$ and $M(z)^{-1}$ are bounded for $z\in\dC\backslash\dR$. For this it suffices to check that the operators
\begin{equation}\label{ees}
 \alpha\cE_{1/2}(z)- I_{L^2(\cC)} \, \text{ and } \, c\cE_{1/2}(z)- I_{L^2(\cC)}
\end{equation}
are bounded and have bounded inverses. The argument is the same for both operators in \eqref{ees} 
and hence we discuss $\alpha\cE_{1/2}(z)- I_{L^2(\cC)}$ only. One recalls that
\begin{equation*}
 \cD_{\rm i}(z) +\cD_{\rm e}(z),\quad z\in\dC\backslash \dR,
\end{equation*}
maps onto $L^2(\cC)$, is boundedly invertible, and its inverse $\cE(z)$ in \eqref{ee} is a compact operator in $L^2(\cC)$ with $\ran(\cE(z))=H^1(\cC)$ (see   
\cite[Proposition~3.2~$(iii)$]{BLL13-AHP}). Hence also the restriction $\cE_{1/2}(z)$ of $\cE(z)$ onto $H^{1/2}(\cC)$ is bounded in $L^2(\cC)$. It follows that
$\alpha\cE_{1/2}(z)- I_{L^2(\cC)}$ is bounded, and its closure is given by
\begin{equation}\label{ovm}
\overline{\alpha \cE_{1/2}(z)- I_{L^2(\cC)}}=\alpha\cE(z)- I_{L^2(\cC)}\in\cL\bigl(L^2(\cC)\bigr), 
\quad z\in\dC\backslash\dR.
\end{equation}

In order to show that the inverse $(\alpha\cE_{1/2}(z)- I_{L^2(\cC)})^{-1}$ exists and is 
bounded for $z\in\dC\backslash \dR$ we first check that 
\begin{equation}\label{kern}
\ker\bigl(\alpha\cE(z)- I_{L^2(\cC)}\bigr)=\{0\},\quad z\in\dC\backslash\dR.
\end{equation}
In fact, assume that $z\in\dC\backslash\dR$ and $\varphi\in L^2(\cC)$
are such that $\alpha\cE(z)\varphi=\varphi$. 
It follows from  $\dom(\cE(z))=\ran(\cD_{\rm i}(z) +\cD_{\rm e}(z))=L^2(\cC)$ that there exists $\psi\in H^1(\cC)$ such that
\begin{equation}\label{zuzu}
 \varphi=\bigl(\cD_{\rm i}(z) +\cD_{\rm e}(z)\bigr)\psi, 
\end{equation}
and from \eqref{i-bvp}--\eqref{e-bvp} one concludes that there exists a unique 
\begin{equation*}
 f_z =\begin{pmatrix} f_{{\rm i},z} \\ f_{{\rm e},z}\end{pmatrix} \in H^{3/2}(\Omega_{\rm i})\times H^{3/2}(\Omega_{\rm e}) 
\end{equation*}
such that
\begin{equation}\label{df1}
\gamma_D^{\rm i}f_{{\rm i},z}=\gamma_D^{\rm e}f_{{\rm e},z}=\psi,
\end{equation}
and
\begin{equation*}
-\Delta f_{{\rm i},z} = z f_{{\rm i},z}\, \text{ and } \, 
-\Delta f_{{\rm e},z} = z f_{{\rm e},z}.  
\end{equation*}
Since $\varphi=\alpha \cE(z)\varphi=\alpha\psi$ by \eqref{zuzu}, one obtains from \eqref{useit2}, \eqref{df1}, and \eqref{zuzu} that
\begin{equation}\label{bcc}
\begin{split}
 \gamma_N^{\rm i}f_{{\rm i},z}+\gamma_N^{\rm e}f_{{\rm e},z}
 &=\bigl(\cD_{\rm i}(z)+\cD_{\rm e}(z)\bigr)\gamma_D^{\rm i}f_{{\rm i},z}\\
 &=\bigl(\cD_{\rm i}(z)+\cD_{\rm e}(z)\bigr)\psi\\
 &=\varphi\\
 &=\alpha\psi\\
 &=\alpha\gamma_D^{\rm i}f_{{\rm i},z}.
\end{split}
\end{equation}
For $h=(h_{\rm i},h_{\rm e})^\top\in\dom(H_{\delta,\alpha})$ one has  
\begin{equation}\label{df2}
 \gamma_D^{\rm i}h_{\rm i}=\gamma_D^{\rm e}h_{\rm e}\, \text{ and } \, 
 \gamma_N^{\rm i}h_{\rm i}+\gamma_N^{\rm e}h_{\rm e} 
 =\alpha\gamma_D^{\rm i}h_{\rm i}, 
\end{equation}
and in a similar way as in \eqref{fvb}, Green's identity together with \eqref{df1}, \eqref{bcc}, and \eqref{df2} imply
\begin{equation*}
\begin{split}
 (H_{\delta,\alpha}&h,f_z )_{L^2(\dR^n)}-(h,z f_z )_{L^2(\dR^n)}\\
&\quad =(-\Delta h_{\rm i},f_{{\rm i},z})_{L^2(\Omega_{\rm i})}-(h_{\rm i},-\Delta f_{{\rm i},z})_{L^2(\Omega_{\rm i})}\\
&\qquad
            + (-\Delta h_{\rm e},f_{{\rm e},z})_{L^2(\Omega_{\rm e})}-(h_{\rm e},-\Delta f_{{\rm e},z})_{L^2(\Omega_{\rm e})}\\
&\quad=(\gamma_D^{\rm i}h_{\rm i},\gamma_N^{\rm i}f_{{\rm i},z})_{L^2(\cC)}-(\gamma_N^{\rm i}h_{\rm i},\gamma_D^{\rm i}f_{{\rm i},z})_{L^2(\cC)}\\       
&\qquad +(\gamma_D^{\rm e}h_{\rm e},\gamma_N^{\rm e}f_{{\rm e},z})_{L^2(\cC)}-(\gamma_N^{\rm e}h_{\rm e},\gamma_D^{\rm e}f_{{\rm e},z})_{L^2(\cC)}  \\
&\quad =\bigl(\gamma_D^{\rm i}h_{\rm i},\gamma_N^{\rm i}f_{{\rm i},z}+\gamma_N^{\rm e}f_{{\rm e},z}\bigr)_{L^2(\cC)}
 -\bigl(\gamma_N^{\rm i}h_{\rm i}+\gamma_N^{\rm e}h_{\rm e},\gamma_D^{\rm i}f_{{\rm i},z}\bigr)_{L^2(\cC)}\\
&\quad =\bigl(\gamma_D^{\rm i}h_{\rm i},\alpha\gamma_D^{\rm i}f_{{\rm i},z}\bigr)_{L^2(\cC)}
 -\bigl(\alpha\gamma_D^{\rm i}h_{\rm i},\gamma_D^{\rm i}f_{{\rm i},z}\bigr)_{L^2(\cC)}\\ 
&\quad=0.
\end{split}
\end{equation*}
As $H_{\delta,\alpha}$ is self-adjoint one concludes that $f_z \in\dom (H_{\delta,\alpha})$ and 
$$f_z \in\ker(H_{\delta,\alpha} - z I_{L^2(\bbR^n)}).$$ 
Since $z\in\dC\backslash\dR$, 
this yields $f_z =0$ and therefore, $\psi= \gamma_D^{\rm i}f_{{\rm i},z}=0$ and hence $\varphi=0$ by \eqref{zuzu}, implying \eqref{kern}.

Since $\cE(z)$ is a compact operator in $L^2(\cC)$ (see 
\cite[Proposition~3.2\,$(iii)$]{BLL13-AHP}) also $\alpha\cE(z)$ is compact 
and together with \eqref{kern} one concludes that 
\begin{equation}\label{vvvv}
(\alpha\cE(z)-I_{L^2(\cC)})^{-1}\in\cL\big(L^2(\cC)\big).
\end{equation}
Hence also the restriction
\begin{equation*}
 \bigl(\alpha\cE_{1/2}(z)-I_{L^2(\cC)}\bigr)^{-1}
\end{equation*}
is a bounded operator in $L^2(\cC)$. Summing up, we have shown that the operators in \eqref{ees} are bounded and have bounded inverses for all $z\in\dC\backslash\dR$,
and hence the values $M(z)$ of the Weyl function in \eqref{mddf} are bounded and have bounded inverses for all $z\in\dC\backslash\dR$. 
From \eqref{mddf}, \eqref{ovm} and \eqref{vvvv} it follows that that
the closures
of the operators $M(z)$, $z\in\dC\backslash\dR$, in $L^2(\cC)$ are given by the operators $\cM_\alpha(z)$ in \eqref{4.5a}, \eqref{4.5b}.

\vskip 0.2cm\noindent
{\it Step 3.} Now we check that the operators $\{H_{\delta,c},H_{\delta,\alpha}\}$ and the Weyl function corresponding to the 
quasi boundary triple $\{L^2(\cC),\Gamma_0,\Gamma_1\}$ in Step 1 satisfy the assumptions of Theorem~\ref{mainssf2} for $n \in \bbN$, $n \geq 2$, and all $k\geq (n-3)/4$.

In fact, the sign condition \eqref{sign333} follows from the assumption $\alpha(x)<c$ and the fact that the closed quadratic forms 
$\mathfrak h_{\delta,\alpha}$ and $\mathfrak h_{\delta,c}$ associated to $H_{\delta,\alpha}$ and $H_{\delta,c}$ satisfy the inequality
$\mathfrak h_{\delta,c}\leq \mathfrak h_{\delta,\alpha}$. More precisely, the inequality for the quadratic forms yields
$\inf(\sigma(H_{\delta,c}))\leq \inf(\sigma(H_{\delta,\alpha}))$, and for
$\zeta < \inf (\sigma(H_{\delta,c}))$ the forms $\mathfrak h_{\delta,c}-\zeta$ and  $\mathfrak h_{\delta,\alpha}-\zeta$ are both nonnegative, satisfy 
the inequality $\mathfrak h_{\delta,c}-\zeta \leq \mathfrak h_{\delta,\alpha}-\zeta$, and hence the resolvents of the
corresponding nonnegative self-adjoint operators $H_{\delta,c}-\zeta I_{L^2(\dR^n)}$ and 
$H_{\delta,\alpha}-\zeta I_{L^2(\dR^n)}$ satisfy the inequality
\begin{equation*}
 (H_{\delta,c}-\zeta I_{L^2(\dR^n)})^{-1}\geq (H_{\delta,\alpha}-\zeta I_{L^2(\dR^n)})^{-1}, 
 \quad \zeta < \inf (\sigma(H_{\delta,c})) 
\end{equation*}
(see, e.g., \cite[Chapter~VI, $\S$~2.6]{K80} or \cite[Chapter~10, $\S$2, Theorem~6]{BS87}). 
Thus the  sign condition \eqref{sign333} in the assumptions of Theorem~\ref{mainssf2}
holds.

In order to verify the $\sS_p$-conditions 
\begin{align}
& \overline{\gamma(z)}^{(p)}\bigl( M(z)^{-1} \gamma({\ol z})^*\bigr)^{(q)}\in\sS_1\bigl(L^2(\dR^n)\bigr),\quad p+q=2k,  \label{ddc1} \\
& \bigl( M(z)^{-1} \gamma({\ol z})^*\bigr)^{(q)}\overline{\gamma(z)}^{(p)}\in\sS_1\bigl(L^2(\cC)\bigr),\quad p+q=2k,  \label{ddc2}
\end{align}
and 
 \begin{equation}\label{ddc3}
   \frac{d^j}{dz^j} \overline{M (z)}\in\sS_{(2k+1)/j}\bigl(L^2(\cC)\bigr),\quad j=1,\dots,2k+1,
 \end{equation}
for all $z \in \rho(H_{\delta,c})\cap\rho(H_{\delta,\alpha})$ in the assumptions of 
Theorem~\ref{mainssf2}, one first recalls the smoothing property
\begin{equation}\label{smoothi2}
 (H_{\delta,c} - z I_{L^2(\bbR^n)})^{-1} f \in H^{k+2}(\Omega_{\rm i})\times H^{k+2}(\Omega_{\rm e})
\end{equation}
for $f\in H^k(\Omega_{\rm i})\times H^k(\Omega_{\rm e})$ and $k \in \bbN_0$, which follows, for instance, from \cite[Theorem~4.20]{McL00}. Next one observes that \eqref{gstar}, \eqref{qbtw2}, 
and the definition of $H_{\delta,c}$ imply 
\begin{equation*}
\begin{split}
\gamma({\ol z})^*f&=\Gamma_1(H_{\delta,c} - z I_{L^2(\bbR^n)})^{-1}f\\
&=(c-\alpha)^{-1}\bigl(\alpha\gamma_D^{\rm i} -(\gamma_N^{\rm i}+\gamma_N^{\rm e})\bigr)(H_{\delta,c} - z I_{L^2(\bbR^n)})^{-1}f\\
&=(c-\alpha)^{-1}\bigl(c\gamma_D^{\rm i} -(\gamma_N^{\rm i}+\gamma_N^{\rm e}) + (\alpha- c)\gamma_D^{\rm i}\bigr)(H_{\delta,c} - z I_{L^2(\bbR^n)})^{-1}f,
\end{split}
\end{equation*}
which yields
\begin{equation}\label{plmplm}
\gamma({\ol z})^*f =-\gamma_D^{\rm i}(H_{\delta,c} - z I_{L^2(\bbR^n)})^{-1}f,\quad f\in L^2(\dR^n).
\end{equation}
Hence \eqref{gammad2}, \eqref{smoothi2}, and Lemma~\ref{usel} imply
\begin{equation}\label{klklkl}
\bigl(\gamma({\ol z})^*\bigr)^{(q)} 
 =-q! \, \gamma_D^{\rm i}(H_{\delta,c} - z I_{L^2(\bbR^n)})^{-(q+1)}\in\sS_r\bigl(L^2(\dR^n),L^2(\cC)\bigr)
\end{equation}
for $r> (n-1)/[2q+(3/2)]$,  $z \in \rho(H_{\delta,c})$ and $q \in \bbN_0$ 
(cf.\ \cite[Lemma~3.1]{BLL-Exner} for the case $c=0$). One also has
\begin{equation}\label{oppo}
 \overline{\gamma(z)}^{(p)}\in\sS_r\bigl(L^2(\cC),L^2(\dR^n)\bigr),\quad r> (n-1)/[2p+(3/2)],
\end{equation}
for all $z \in \rho(H_{\delta,c})$ and $p \in \bbN_0$. 
Furthermore,
\begin{equation}\label{mqqq}
   \frac{d^j}{dz^j} \overline{M(z)}
   = j! \, \gamma({\ol z})^*(H_{\delta,c} - z I_{L^2(\bbR^n)})^{-(j-1)} 
   \overline{\gamma(z)}
 \end{equation}
by \eqref{gammad3} and with the help of \eqref{plmplm} it follows that 
\begin{equation*}
\gamma({\ol z})^*(H_{\delta,c} - z I_{L^2(\dR^n)})^{-(j-1)} 
= - \gamma_D^{\rm i} (H_{\delta,c} - z I_{L^2(\dR^n)})^{-j} \in 
\sS_x\bigl(L^2(\dR^n),L^2(\cC)\bigr)
 \end{equation*}
for $x> (n-1)/[2j-(1/2)]$. Moreover, we have $\overline{\gamma(z)}\in\sS_y(L^2(\cC),L^2(\dR^n))$ for $y> 2 (n-1)/3$ by \eqref{oppo} and hence 
it follows from \eqref{mqqq} and the well-known property $PQ\in\sS_w$ for $P\in\sS_x$, $Q\in\sS_y$, and $x^{-1}+y^{-1}=w^{-1}$, that 
 \begin{equation}\label{ass3w77}
 \frac{d^j}{dz^j} \overline{M(z)}\in\sS_w\bigl(L^2(\cC)\bigr), 
 \quad w >  (n-1)/(2j+1), \; z \in \rho(H_{\delta,c}), \; j \in \bbN.
\end{equation}
One observes that
\begin{equation*}
 \frac{d}{dz} \big[\overline{M(z)}\big]^{-1} = - \big[\overline{M(z)}\big]^{-1}\left(\frac{d}{dz} \overline{M(z)}\right)\big[\overline{M(z)}\big]^{-1}, \quad z \in 
 \rho(H_{\delta,c})\cap\rho(H_{\delta,\alpha}),
\end{equation*}
that $\big[\overline{M(z)}\big]^{-1}$ is bounded, and by \eqref{ass3w77} that for $j\in\bbN$ also 
\begin{equation}\label{ass4w44}
 \frac{d^j}{dz^j} \big[\overline{M (z)}\big]^{-1}\in\sS_w\bigl(L^2(\cC)\bigr), 
 \quad w > (n-1)/(2j+1), \; z \in \rho(H_{\delta,c})\cap\rho(H_{\delta,\alpha});
\end{equation}
we leave the formal induction step to the reader. Therefore, 
\begin{equation}\label{zuju}
\begin{split}
&\bigl(  M(z)^{-1} \gamma({\ol z})^* \bigr)^{(q)}=\bigl( \big[\overline{M(z)}\big]^{-1} \gamma({\ol z})^*\bigr)^{(q)} \\
&\qquad=\sum_{\substack{p+m=q \\[0.2ex] p,m\ge0}} \begin{pmatrix} q \\ p \end{pmatrix} 
 \bigl(\big[\overline{M(z)}\big]^{-1}\bigr)^{(p)} \bigl(\gamma({\ol z})^*\bigr)^{(m)}\\
 &\qquad=\big[\overline{M(z)}\big]^{-1} \bigl(\gamma({\ol z})^*\bigr)^{(q)}+ \sum_{\substack{p+m=q \\[0.2ex] p > 0, m\geq 0}} \begin{pmatrix} q \\ p \end{pmatrix} 
 \bigl(\big[\overline{M(z)}\big]^{-1}\bigr)^{(p)} \bigl(\gamma({\ol z})^*\bigr)^{(m)}, 
 \end{split}
 \end{equation}
and one has $$\big[\overline{M(z)}\big]^{-1} (\gamma({\ol z})^*)^{(q)}\in\sS_r\big(L^2(\dR^n),L^2(\cC)\big)$$ for $r> (n-1)/[2q+(3/2)]$ by \eqref{klklkl} and each
summand (and hence also the finite sum) on the right-hand side in \eqref{zuju} is in $\sS_r(L^2(\dR^n),L^2(\cC))$ for 
$r> (n-1)/[2p+1+2m+(3/2)] = (n-1)/[2q+(5/2)]$, which follows from
\eqref{ass4w44} and \eqref{klklkl}. Hence one has 
\begin{equation}\label{hohojaa} 
 \bigl( M(z)^{-1} \gamma({\ol z})^*\bigr)^{(q)}\in \sS_r\bigl(L^2(\dR^n),L^2(\cC)\bigr)
\end{equation}
for $r > (n-1)/[2q+(3/2)]$ and $z \in \rho(H_{\delta,c})\cap\rho(H_{\delta,\alpha})$.
From \eqref{oppo} and \eqref{hohojaa} one then concludes 
\begin{align*}
\overline{\gamma(z)}^{(p)}\bigl( M(z)^{-1} \gamma({\ol z})^*\bigr)^{(q)}\in\sS_r\bigl(L^2(\dR^n)\bigr)
\end{align*}
for $r > (n-1)/[2(p+q)+3] = (n-1)/(4k+3)$,
and since $k\geq (n-3)/4$, one has $1> (n-1)/(4k+3)$, that is, the trace class condition~\eqref{ddc1} is satisfied. The same argument shows that \eqref{ddc2} is satisfied.
Finally, \eqref{ddc3} follows from \eqref{ass3w77} and the fact that $k\geq (n-3)/4$ implies
\begin{equation*}
 \frac{2k+1}{j}\geq\frac{n-1}{2j}>\frac{n-1}{2j+1},\quad j=1,\dots,2k+1.
\end{equation*} 

Hence the assumptions in Theorem~\ref{mainssf2} are satisfied with $S$ in Step 1, the quasi boundary triple 
in \eqref{qbtw1}--\eqref{qbtw2}, the corresponding $\gamma$-field, and Weyl function in \eqref{mddf}. Hence, Theorem~\ref{mainssf2} yields
assertion $(i)$ in Theorem~\ref{dddthm} with $H$ replaced by $H_{\delta,c}$. In addition, for any orthonormal basis $\{\varphi_j\}_{j \in J}$ in $L^2(\cC)$, the function 
  \begin{equation*}
   \xi_\alpha(\lambda) 
   =\sum_{j \in J} \lim_{\varepsilon\downarrow 0}\frac{1}{\pi} \bigl(\Im\bigl( \log (\cM_\alpha(\lambda+i\varepsilon)) \bigr)
   \varphi_j,\varphi_j\bigr)_{L^2(\cC)}  \, \text{ for a.e.~$\lambda \in \dR$},
  \end{equation*}
is a spectral shift function for the pair $\{H_{\delta,c},H_{\delta,\alpha}\}$ such that $\xi_\alpha(\lambda)=0$ for 
$\lambda < \inf(\sigma(H_{\delta,c}))\leq \inf(\sigma(H_{\delta,\alpha}))$ and the trace formula
\begin{equation*}
\begin{split}
& \tr_{L^2(\dR^n)}\bigl( (H_{\delta,\alpha} - z I_{L^2(\dR^n)})^{-(2k+1)} 
 - (H_{\delta,c} - z I_{L^2(\dR^n)})^{-(2k+1)}\bigr) \\
 &\quad= - (2k+1) \int_\dR \frac{\,\xi_\alpha(\lambda)\, d\lambda}{(\lambda - z)^{2k+2}}, \quad z \in \rho(H_{\delta,c})\cap\rho(H_{\delta,\alpha}), 
\end{split}
 \end{equation*}
holds.

The above considerations remain valid in the special case $\alpha=0$ which corresponds
to the pair $\{H_{\delta,c},H\}$ and yields an analogous representation for a spectral shift function $\xi_0$.
Finally it follows from the considerations in the end of Section~\ref{ssfsec} (see \eqref{ssfab}) 
that 
\begin{equation*}
\begin{split}
   \xi(\lambda)&=\xi_\alpha(\lambda)-\xi_0(\lambda)\\
   &=\sum_{j \in J} \lim_{\varepsilon\downarrow 0}\frac{1}{\pi} \Bigl(\bigl(\Im\bigl( \log (\cM_\alpha(\lambda+i\varepsilon)) - \log (\cM_0(\lambda + i \varepsilon))\bigr)\bigr)
   \varphi_j,\varphi_j\Bigr)_{L^2(\cC)} 
\end{split}
   \end{equation*}
   for a.e.~$\lambda \in \dR$
is a spectral shift function for the pair $\{H,H_{\delta,\alpha}\}$ such that $\xi(\lambda)=0$ for 
$\lambda < \inf(\sigma(H_{\delta,c}))\leq \inf\{\sigma(H),\sigma(H_{\delta,\alpha})\}$. This completes the proof of Theorem~\ref{dddthm}.
\end{proof}

In space dimensions $n=2$ and $n=3$ one can choose $k=0$ in Theorem~\ref{dddthm} and together with Corollary~\ref{mainthmcorchen} one obtains the following result.

\begin{corollary}\label{mainthmcorchen3}
Let the assumptions and $\cM_\alpha$ and $\cM_0$ be as in Theorem~\ref{dddthm}, and suppose that $n=2$ or $n=3$. 
Then the following assertions $(i)$--$(iii)$ hold: 
\begin{itemize}
  \item [$(i)$]
The difference of the resolvents of $H$ and $H_{\delta,\alpha}$ is 
a trace class operator, that is, for all $z\in\rho(H_{\delta,\alpha})=\rho(H)\cap\rho(H_{\delta,\alpha})$, 
\begin{equation*}
 \big[(H_{\delta,\alpha} - z I_{L^2(\dR^n)})^{-1}-(H - z I_{L^2(\dR^n)})^{-1}\big] \in \sS_1\bigl(L^2(\dR^n)\bigr).
\end{equation*}
  \item [$(ii)$] $\Im(\log (\cM_\alpha(z)))\in\sS_1(L^2(\cC))$ and $\Im(\log (\cM_0(z)))\in\sS_1(L^2(\cC))$ for all $z\in\dC\backslash\dR$, and the limits 
  $$\Im\bigl(\log(\cM_\alpha(\lambda+i 0))\bigr):=\lim_{\varepsilon\downarrow 0}\Im\bigl(\log(\cM_\alpha(\lambda+i\varepsilon))\bigr)$$ 
  and
  $$\Im\bigl(\log(\cM_0(\lambda+i 0))\bigr):=\lim_{\varepsilon\downarrow 0}\Im\bigl(\log(\cM_0(\lambda+i\varepsilon))\bigr)$$ 
  exist for a.e.~$\lambda \in \dR$ in $\sS_1(L^2(\cC))$. 
  \item [$(iii)$] The function
  \begin{equation*}
   \xi(\lambda)=\frac{1}{\pi} \tr_{L^2(\cC)}\bigl(\Im\big(\log(\cM_\alpha(\lambda + i0))-\log(\cM_0(\lambda + i0)) \big)\bigr) 
  \end{equation*}
  for a.e. $\lambda\in\dR$,
is a spectral shift function for the pair $\{H,H_{\delta,\alpha}\}$ such that $\xi(\lambda)=0$ for 
$\lambda < \inf(\sigma(H_{\delta,c}))$ and the trace formula
\begin{equation*}
 \tr_{L^2(\dR^n)}\bigl( (H_{\delta,\alpha} - z I_{L^2(\dR^n)})^{-1} 
 - (H - z I_{L^2(\dR^n)})^{-1}\bigr) 
 = -  \int_\dR \frac{\,\xi(\lambda)\, d\lambda}{(\lambda - z)^2}
\end{equation*}
is valid for all $z\in\rho(H_{\delta,\alpha})=\rho(H)\cap\rho(H_{\delta,\alpha})$.
\end{itemize}
\end{corollary}

In the special case $\alpha<0$, Theorem~\ref{dddthm} simplifies slightly since in that case the sign condition \eqref{sign333} in Theorem~\ref{mainssf2}
is satisfied by the pair $\{H,H_{\delta,\alpha}\}$. Hence it is not necessary to introduce the operator $H_{\delta,c}$ as a comparison
operator in the proof of Theorem~\ref{dddthm}. Instead, one considers the operators $S$ and $T$ in Step 1 of the proof of Theorem~\ref{dddthm}, and defines the 
boundary maps by
\begin{equation*}
 \Gamma_0 f=-\gamma_N^{\rm i}f_{\rm i}-\gamma_N^{\rm e}f_{\rm e},\quad \dom(\Gamma_0)=\dom(T),
\end{equation*}
and
\begin{equation*}
 \Gamma_1 f=-\gamma_D^{\rm i} f_{\rm i}+\frac{1}{\alpha}(\gamma_N^{\rm i}f_{\rm i}+\gamma_N^{\rm e}f_{\rm e})\bigr),\quad 
 \dom(\Gamma_1)=\dom(T).
\end{equation*} 
In this case the corresponding Weyl function is given by
\begin{equation*}
 M(z)=\cE_{1/2}(z) - \alpha^{-1} I_{L^2(\cC)},\quad z\in\dC\backslash\dR,
\end{equation*}
and hence the next statement follows in the same way as Theorem~\ref{dddthm} from our abstract result Theorem~\ref{mainssf2}. 

\begin{theorem}\label{dddthm2}
Assume Hypothesis~\ref{hypo6}, 
let $\cE(z)$ be defined as in \eqref{ee}, and let $\alpha\in C^1(\cC)$ be a real-valued function such that $\alpha(x)<0$ for all $x\in\cC$. 
Then the following assertions $(i)$ and $(ii)$ hold for $k \in \bbN_0$ such that $k\geq (n-3)/4$:
 \begin{itemize}
  \item [$(i)$] The difference of the $2k+1$th-powers of the resolvents of $H$ and $H_{\delta,\alpha}$ is 
a trace class operator, that is,
\begin{equation*}
\big[(H_{\delta,\alpha} - z I_{L^2(\bbR^n)})^{-(2k+1)} 
- (H - z I_{L^2(\bbR^n)})^{-(2k+1)}\big] \in \sS_1\bigl(L^2(\dR^n)\bigr)
\end{equation*}
holds for all $z\in\rho(H_{\delta,\alpha})=\rho(H)\cap\rho(H_{\delta,\alpha})$. 
 \item [$(ii)$] For any orthonormal basis $(\varphi_j)_{j \in J}$ in $L^2(\cC)$ the function 
  \begin{equation*}
   \xi(\lambda) 
   =\sum_{j \in J} \lim_{\varepsilon\downarrow 0}\frac{1}{\pi} \bigl(\Im\bigl( \log (\cE(t+i\varepsilon) - \alpha^{-1} I_{L^2(\cC)}) \bigr)
   \varphi_j,\varphi_j\bigr)_{L^2(\cC)} 
  \end{equation*}
  for a.e. $\lambda\in\dR$, 
is a spectral shift function for the pair $\{H,H_{\delta,\alpha}\}$ such that $\xi(\lambda)=0$ for 
$\lambda < 0$ and the trace formula
\begin{align*}
 \tr_{L^2(\bbR^n)}\bigl( (H_{\delta,\alpha} & - z I_{L^2(\bbR^n)})^{-(2k+1)} 
 - (H - z I_{L^2(\bbR^n)})^{-(2k+1)}\bigr)    \\ 
 & \qquad = - (2k+1) \int_\dR \frac{\xi(\lambda)\, d\lambda}{(\lambda - z)^{2k+2}} 
 \end{align*}
is valid for all $z\in\dC\backslash [0,\infty)$.
\end{itemize}
\end{theorem}

The analog of Corollary~\ref{mainthmcorchen3} again holds in the special cases $n=2$ and $n=3$; we omit further details.

\vskip 0.8cm
\noindent {\bf Acknowledgments.}  
J.B.\ is most grateful for the stimulating research stay and the hospitality at the 
Graduate School of Mathematical Sciences of the University of Tokyo from April to July 2016, where parts of this paper were written. F.G.\ is indebted to all organizers of the IWOTA 2017 Conference for creating such a stimulating atmosphere and for the great hospitality in Chemnitz, Germany, August 14--18. 
The authors also wish to thank Hagen Neidhardt for fruitful discussions and helpful remarks. This work is supported by International Relations and Mobility Programs of the TU Graz and the Austrian Science Fund (FWF), project P~25162-N26.



\begin{thebibliography}{99}

\bibitem{AS72} 
M.~Abramowitz and I.~A.~Stegun, {\it Handbook of Mathematical
Functions}, Dover, New York, 1972.

\bibitem{AGHH05}
S.~Albeverio, F.~Gesztesy, R.~H{\o}egh-Krohn, and H.~Holden, 
{\it Solvable Models in Quantum Mechanics}, 2nd edition. With an appendix by Pavel Exner.
AMS Chelsea Publishing, Providence, RI, 2005.

\bibitem{AKMN13}
S.~Albeverio, A.~Kostenko, M.\,M.~Malamud, and H.~Neidhardt,
Spherical Schr\"odinger operators with $\delta$-type interactions, 
J. Math. Phys. {\bf 54} (2013), 052103. 

\bibitem{AK99}
S.~Albeverio and P.~Kurasov, 
{\it Singular Perturbations of Differential Operators}, 
London Mathematical Society Lecture Note Series, Vol.\ {\bf 271}, Cambridge University Press, Cambridge, 2000.

\bibitem{AGS87} J.-P.~Antoine, F.~Gesztesy, and J.~Shabani,
Exactly solvable models of sphere interactions in quantum mechanics,
J.\ Phys.\ A {\bf 20} (1987), 3687--3712.

\bibitem{BGN17} J.~Behrndt, F. Gesztesy, and S.~Nakamura, 
Spectral shift functions and Dirichlet-to-Neumann maps, to appear in Math. Ann.

\bibitem{BL07} J.~Behrndt and M.~Langer, 
Boundary value problems for elliptic partial differential operators on bounded domains, 
J. Funct. Anal. {\bf 243} (2007), 536--565.

\bibitem{BL12} J.~Behrndt and M.~Langer,
Elliptic operators, Dirichlet-to-Neumann maps and quasi boundary triples,
in: \textit{Operator Methods for Boundary Value Problems},
London Math.\ Soc.\ Lecture Note Series, Vol.\ {\bf 404}, 2012, pp.~121--160.

\bibitem{BLL13-AHP} J.~Behrndt, M.~Langer, and V.~Lotoreichik, 
Schr\"odinger operators with $\delta$ and $\delta'$-potentials supported on hypersurfaces, 
Ann.\ Henri Poincar\'e {\bf 14} (2013), 385--423.

\bibitem{BLL13-IEOT} J.~Behrndt, M.~Langer, and V.~Lotoreichik,
Spectral estimates for resolvent differences of self-adjoint elliptic operators, Integral Equations Operator Theory {\bf 77} (2013), 1--37.

\bibitem{BLL13-3} J.~Behrndt, M.~Langer, and V.~Lotoreichik, 
Trace formulae and singular values of resolvent power differences of self-adjoint elliptic operators,
J. London Math. Soc. {\bf 88} (2013), 319--337.

\bibitem{BLL-Exner} J.~Behrndt, M.~Langer, and V.~Lotoreichik, 
Trace formulae for Schr\"{o}dinger operators with singular interactions,
in {\it Functional Analysis and Operator Theory for Quantum Physics}, 
J.\ Dittrich, H.\ Kovarik, and A.\ Laptev (eds.), EMS Publishing House, EMS, 
ETH--Z\"urich, Switzerland (to appear). 

\bibitem{BMN17} J.~Behrndt, M.\,M.~Malamud, and H.~Neidhardt, Scattering matrices and Dirichlet-to-Neumann maps, J. Funct. Anal. {\bf 273} (2017), 1970--2025. 

\bibitem{BP98}
M.\,Sh.~Birman and A.\,B.~Pushnitski,  
Spectral shift function, amazing and multifaceted, 
Integral Equations Operator Theory {\bf 30} (1998), 191--199.

\bibitem{BS87} 
M.\,Sh.~Birman and M.\,Z.~Solomjak,
\textit{Spectral Theory of Self-Adjoint Operators in Hilbert Spaces},
D.\ Reidel Publishing Co., Dordrecht, 1987.

\bibitem{BY92}
M.\,Sh.~Birman and D.\,R.~Yafaev, 
The spectral shift function. The papers of M. G. Krein and their further development, 
Algebra i Analiz {\bf 4} (1992), no. 5 1--44; translation in St. Petersburg Math. J. 
{\bf 4} (1993), no. 5, 833--870. 

\bibitem{BY92-1}
M.\,Sh.~Birman and D.\,R.~Yafaev, 
Spectral properties of the scattering matrix, 
Algebra i Analiz {\bf 4} (1992), no. 6 1--27; translation in St. Petersburg 
Math. J. {\bf 4} (1993), no. 6, 1055--1079.

\bibitem{BEKS94}
J.\,F.~Brasche, P.~Exner, Yu.\,A.~Kuperin, and P.~Seba,
Schr\"odinger operators with singular interactions,
J. Math. Anal. Appl. {\bf 184} (1994), 112--139.

\bibitem{BGP08}  J.~Br\"uning, V.~Geyler, and K.~Pankrashkin, 
Spectra of self-adjoint extensions and applications to solvable Schr\"odinger operators, 
Rev. Math. Phys. {\bf 20} (2008), 1--70.

\bibitem{DM91} V.\,A.~Derkach and M.\,M.~Malamud,
Generalized resolvents and the boundary value problems for Hermitian operators
with gaps, J.\ Funct.\ Anal. {\bf 95} (1991), 1--95.

\bibitem{DM95} V.\,A.~Derkach and M.\,M.~Malamud,
The extension theory of Hermitian operators and the moment problem,
J.\ Math.\ Sci.\ (NY) {\bf 73} (1995), 141--242.

\bibitem{E08} P.~Exner,
Leaky quantum graphs: a review, Proc.\ Symp.\ Pure Math. {\bf 77} (2008), 523--564.

\bibitem{EI01}
P.~Exner and T.~Ichinose,
Geometrically induced spectrum in curved leaky wires,
J.\ Phys.\ A {\bf 34} (2001), 1439--1450.

\bibitem{EK03} P.~Exner and S.~Kondej,
Bound states due to a strong $\delta$ interaction supported by a curved surface,
J.\ Phys.\ A {\bf 36} (2003), 443--457.

\bibitem{EK05}
P.~Exner and S.~Kondej,
Scattering by local deformations of a straight leaky wire,
J.\ Phys.\ A {\bf 38} (2005) 4865--4874.

\bibitem{EK15}
P.~Exner and H.~Kova\v{r}\'{\i}k, {\it Quantum Waveguides}, Cham, Springer, 2015.

\bibitem{EY02}
P.~Exner and K.~Yoshitomi,
Asymptotics of eigenvalues of the Schr\"odinger operator with a strong
$\delta$-interaction on a loop,
J.\ Geom.\ Phys. {\bf 41} (2002), 344--358.

\bibitem{GMN99}
F.~Gesztesy, K.\,A.~Makarov, and S.\,N.~Naboko, 
The spectral shift operator, 
Operator Theory Advances Applications {\bf 108} (1999), 59--90.

\bibitem{GG91} V.\,I.~Gorbachuk and M.\,L.~Gorbachuk,
{\it Boundary Value Problems for Operator Differential Equations},
Kluwer Academic Publishers, Dordrecht, 1991.
 
\bibitem{K80} T.~Kato, {\it Perturbation Theory for Linear Operators}, 
Grundlehren der mathematischen Wissenschaften, Vol.\ {\bf 132}, 
corr.\ printing of the 2nd ed., Springer, Berlin, 1980.

\bibitem{K53} M.\,G.~Krein,
On the trace formula in perturbation theory, 
Mat. Sbornik {\bf 33} (1953), 597--626.

\bibitem{K62} M.\,G.~Krein,
On perturbation determinants and a trace formula for unitary and self-adjoint operators, 
Dokl. Akad. Nauk SSSR {\bf 144} (1962), 268--271.

\bibitem{L52} I.\,M.~Lifshitz, 
On a problem of the theory of perturbations connected with quantum statistics, 
Uspehi Matem. Nauk {\bf 7} (1952), 171--180.

\bibitem{L56} I.\,M.~Lifshits, Some problems of the dynamic
theory of nonideal crystal lattices, Nuovo Cimento Suppl.
{\bf 3} (Ser.~X) (1956), 716--734.

\bibitem{MPS15} 
A.~Mantile, A.~Posilicano, and M.~Sini,  
Self-adjoint elliptic operators with boundary conditions on not closed hypersurfaces, 
J. Diff. Eq. {\bf 261} (2016), 1--55.

\bibitem{McL00} 
W.~McLean,
\textit{Strongly Elliptic Systems and Boundary Integral Equations}, 
Cambridge University Press, 2000.

\bibitem{S12}
K.~Schm\"{u}dgen,
{\it Unbounded Self-Adjoint Operators on Hilbert Space},
Springer, Dordrecht, 2012.

\bibitem{Y92} D.\,R.~Yafaev,
{\it Mathematical Scattering Theory. General Theory},
Translations of Mathematical Monographs, Vol.\ {\bf 105}. Amer. Math. Soc., 
Providence, RI, 1992.

\bibitem{Y05}
D.\,R.~Yafaev, 
A trace formula for the Dirac operator, 
Bull. London Math. Soc. {\bf 37} (2005), 908--918.

\bibitem{Y10} D.\,R.~Yafaev,
{\it Mathematical Scattering Theory. Analytic Theory}, 
Mathematical Surveys and Monographs, Vol.\ {\bf 158}, Amer. Math. Soc., 
Providence, RI, 2010.

\end{thebibliography}
\end{document}